\documentclass[letterpaper,11pt]{amsart}

\usepackage[margin=1.2in]{geometry}
\usepackage{mathrsfs,amsmath,amsthm,amssymb}
\usepackage{xspace,xcolor}
\usepackage[breaklinks,colorlinks,citecolor=teal,linkcolor=teal,urlcolor=teal,pagebackref,hyperindex]{hyperref}
\usepackage[alphabetic]{amsrefs}
\usepackage[all,cmtip]{xy}
\usepackage{mathtools}
\hypersetup{final}
\numberwithin{equation}{section}

\newtheorem*{theorem*}{Theorem}
\newtheorem{theorem}{Theorem}[section]
\newtheorem{lemma}[theorem]{Lemma}
\newtheorem{proposition}[theorem]{Proposition}
\newtheorem{corollary}[theorem]{Corollary}
\newtheorem{question}[theorem]{Question}
\newtheorem{conjecture}[theorem]{Conjecture}
\newtheorem{step}{Step}

\theoremstyle{definition}
\newtheorem{definition}[theorem]{Definition}
\newtheorem{example}[theorem]{Example}

\theoremstyle{remark}
\newtheorem*{remark}{Remark}
\newtheorem{case}{Case}
\newtheorem{claim}{Claim}
\newtheorem*{ack}{Acknowledgments}

\newenvironment{hproof}{%
  \proof}{\endproof}

\DeclareMathOperator{\reg} {reg}

\DeclareMathOperator{\Sym}{Sym}

\DeclareMathOperator{\Eff}{Eff}

\DeclareMathOperator{\Supp}{Supp}

\DeclareMathOperator{\Big1}{Big}
\DeclareMathOperator{\N}{N}

\DeclareMathOperator{\Bl}{Bl}
\DeclareMathOperator{\Mov}{Mov}
\DeclareMathOperator{\Spec}{Spec}
\DeclareMathOperator{\WMov}{WMov}
\DeclareMathOperator{\Nef}{Nef}
\DeclareMathOperator{\cd}{cd}
\DeclareMathOperator{\rk}{rk}
\DeclareMathOperator{\Z}{Z}
\DeclareMathOperator{\Pic}{Pic}

\title[On nef subvarieties]{On nef subvarieties}
\author{Chung Ching Lau}
\date{}
\address{Department of Mathematics, University of Utah, Salt Lake City, UT 84112, USA}
\email{{\tt lau@math.utah.edu}}
\thanks{Research partially suppported by NSF FRG grant DMS-1265285}
\begin{document}

\maketitle
\begin{abstract}
In this paper, we first study generalizations of Fujita vanishing theorems for $q$-ample divisors.
We then apply them to study positivity of subvarieties with nef normal bundle in the sense of intersection theory.

    After Ottem's work on ample subschemes, we introduce the notion of a nef subscheme, which generalizes the notion of a subvariety with nef normal bundle.
    We show that restriction of a pseudoeffective (resp. big) divisor to a nef subvariety is pseudoeffective (resp. big).
    We also show that ampleness and nefness are transitive properties.
  
    We define the weakly movable cone as the
cone generated by the pushforward of cycle classes of nef subvarieties via
proper surjective maps. This cone contains the movable cone
and shares similar intersection-theoretic properties with it, thanks to the
aforementioned properties of nef subvarieties.
    
    On the other hand, we prove that if $Y\subset X$ is an ample subscheme of codimension $r$ and $D|_Y$ is $q$-ample, then $D$ is $(q+r)$-ample. This is analogous to a result proved by Demailly-Peternell-Schneider and K\"uronya.
\end{abstract}

\section{Introduction}

The concept of ampleness of a divisor is central in the subject of algebraic geometry,
connecting vanishing of cohomology groups with geometry.
Weakening the Serre vanishing condition, a line bundle $\mathscr{L}$ is defined to be \textit{$q$-ample} ($q\in\mathbb{N}$) if given any coherent sheaf $\mathscr{F}$, there is a positive integer $m_0$ such that
\[
H^{i}(X,\mathscr{F}\otimes\mathscr{L}^{\otimes m})=0
\]
for $i>q$ and $m>m_0$.
Here we assume $X$ is projective over a field of characteristic zero.
After the works of Andreotti-Grauert \cite{Andreotti}, Sommese \cite{Sommese} and Demailly-Peternell-Schneider \cite{DPS} on $q$-ample divisors, Totaro established the fundamental properties of $q$-ample divisors \cite{Totaro}.
There is another approach to partial ampleness of a line bundle (\cite{dFKL}, \cite{Kuronya1}) that we don't pursue here.

In this paper, we first prove two generalized versions of Fujita vanishing theorem for $q$-ample divisors (Theorem \ref{theorem: fujita} and Proposition \ref{lemma: uniform2}), improving one of the main results in K\"uronya's paper \cite[Theorem C]{Kuronya}.
One of these generalizations (Theorem \ref{theorem: fujita}) can actually be deduced directly from Keeler's generalization of Fujita vanishing theorem \cite[Theorem 4.5 and Theorem 5.5]{Keeler}, combined with Totaro's results on $q$-ample divisors\footnote{We thank one of the referees for pointing this out.}.
Keeler's results are more general.
He took care of the characteristic $p$ case and considered higher rank bundles.
For brevity, we only state a simplified
version of Theorem \ref{theorem: fujita} here.

\begin{theorem}
Let $X$ be a projective scheme of dimension $n$.
Let $\mathscr{L}$ be a $q$-ample line bundle on $X$ and $\mathscr{F}$ be a coherent sheaf on $X$. 
Then there is a positive integer $M$, depending only on $\mathscr{L}$ and $\mathscr{F}$,
such that
\[
H^{i}(X,\mathscr{F}\otimes\mathscr{L}^{\otimes m}\otimes\mathscr{P})=0
\]
for $i>q$, $m\geq M$ and any nef line bundle $\mathscr{P}$ on $Z$.
\end{theorem}

Our second version of Fujita vanishing theorem focuses only on the vanishing of the top cohomology (Proposition \ref{lemma: uniform2}).
In the latter sections, we apply these results to
study positive subvarieties of X.

After the extensive work of Hartshorne \cite{Hartshorne}, where he studied positivity properties of higher codimension subvarieties,
Ottem introduced a
better behaved notion of ampleness for not necessarily l.c.i. subschemes \cite{Ottem}.
He defined a subscheme $Y$ of codimension $r$ of a projective scheme $X$ to be \textit{ample} if the exceptional divisor in the blowup of $X$ along $Y$ is $(r-1)$-ample. 
It is a natural definition that generalizes many properties of ample divisors
\cite[Corollary 5.6]{Ottem}, which were predicted in Hartshorne's work,
while at the same time includes the zero locus of a global section of an ample vector bundle \cite[Proposition 4.5]{Ottem}.
We would also like to note that
Halic studied the notion of partial ampleness of subschemes in \cite{Hal1}, \cite{Hal2}.
He showed that partially ample subvarieties satisfied the G3-property and proved a Fulton-Hansen-type connectedness theorem in the setting of partially ample subvarieties.

Our next result sheds more light on the connection between $q$-ample divisors and ample subschemes:
\begin{theorem}\label{thm: 1}
Let $X$ be a projective scheme of dimension $n$ and $Y$ be an ample subscheme of $X$ of codimension $r$.
Suppose $\mathscr{L}$ is a line bundle on $X$, and that its restriction $\mathscr{L}|_{Y}$ to $Y$ is $q$-ample.
Then $\mathscr{L}$ is $(q+r)$-ample.
\end{theorem}
This result can be compared to a result by Demailly-Peternell-Schneider \cite[Theorem 3.4 and Proposition 1.2]{DPS}.
Given
a chain of codimension $1$ subvarieties
$Y_{n-r}\subset Y_{n-r+1}\subset  \cdots \subset Y_{n-1}\subset Y_n = X$,
such that for $n-r\leq i\leq n-1$, there exists an ample divisor $Z_i$ in the normalization of $Y_{i+1}$,
with $Y_{i}$ being the image of $Z_i$ under the normalization map.
They showed that if $\mathscr{L}|_{Y_{n-r}}$ is ample, then $\mathscr{L}$ is $r$-ample.
A predecessor of Theorem \ref{thm: 1} was also proven by K\"uronya in \cite[Theorem A]{Kuronya}, where he considered the case where $Y$ is a complete intersection of ample divisors.

We now move on to study a weaker positivity condition of a subscheme.
Given an l.c.i. subvariety $Y\subset X$ with nef normal bundle, we would like to understand its positivity properties in terms of intersection theory.
Fulton and Lazarsfeld \cite{FulLaz} gave an answer to this:
They showed that 
given any subvariety $Z\subset X$ such that $\dim Y +\dim Z\geq \dim X$, 
\[
\deg_H(Y\cdot Z)\geq0.
\]
Here $H$ is an ample divisor.

Now let $Y\subset X$ be an arbitrary subscheme of codimension $r$ and $E$ be the exceptional divisor in $\Bl_Y X$.
We say that $Y$ is \textit{nef}
if 
$(E+\epsilon A)|_E$ is $(r-1)$-ample,
where is $A$ is an ample divisor and $0<\epsilon \ll 1$.
This definition is inspired by Ottem's definition of an ample subscheme \cite{Ottem}.
If $Y$ is l.c.i. in $X$, $Y$ is nef if and only if $Y$ has nef 
normal bundle.
We show that 
\begin{theorem}\label{thm: 2}
Let $\iota: Y\hookrightarrow X$ be a nef subvariety of codimension $r$ of a projective variety $X$. 
Then the natural map $\iota^{*}: \N^{1}(X)_{\mathbf{R}} \rightarrow \N^{1}(Y)_{\mathbf{R}}$ induces
$\iota^{*}: \overline{\Eff}^{1}(X)\rightarrow \overline{\Eff}^{1}(Y)$ and
$\iota^{*}: \Big1(X)\rightarrow \Big1(Y)$.
\end{theorem}
When $Y$ is a curve with nef normal sheaf, this is a result of Demailly-Peternell-Schneider \cite[Theorem 4.1]{DPS}.
We also show that nefness and ampleness are transitive properties without any assumptions on singularities, thus generalizes Ottem's result \cite[Proposition 6.4]{Ottem}.
\begin{theorem}\label{thm: 3}
Let $X$ be a projective scheme of dimension $n$.
If $Y$ is an ample (resp. nef) subscheme of $X$ and $Z$ is an ample (resp. nef) subscheme of $Y$, then $Z$ is ample (resp. nef) in $X$.
\end{theorem}

Denote $\N_d(X)_{\mathbf{R}}$ the $d$-cycle classes with real coefficients modulo numerical equivalence.
We define the weakly movable cone, $\overline{\WMov}_d(X)\subset \N_d(X)_{\mathbf{R}}$, as the closure of the convex cone generated by the pushforward of cycle classes of nef subvarieties of dimension $d$ via proper surjective morphisms.
Using the properties stated above on nef subvarieties,
we show that the weakly movable cone shares similar properties to that of the movable cone of $d$-cycles,
$\overline{\Mov}_d(X)$.
\begin{theorem}
Let $X$ be a projective variety of dimension $n$. For $1\leq d \leq n-1$,
\begin{enumerate}
    \item
    $\overline{\Mov}_d(X)\subseteq\overline{\WMov}_d(X)$ and $\overline{\Mov}_1(X)=\overline{\WMov}_1(X)$.
    \item 
    \label{item: eff1}
    $\overline{\Eff}^{1}(X)\cdot \overline{\WMov}_d(X)\subseteq \overline{\Eff}_{d-1}(X)$.
    \item 
    \label{item: bignef1} 
    Let $H$ be a big Cartier divisor, $\alpha \in \overline{\WMov}_d(X)$.
    If $H\cdot \alpha=0$, then $\alpha =0$. 
    \item 
    \label{item: nef1}
    $\Nef^{1}(X)\cdot \overline{\WMov}_d(X)\subseteq\overline{\WMov}_{d-1}(X)$.
\end{enumerate}
\end{theorem}
Analogous statements of (\ref{item: eff1}), (\ref{item: bignef1}) and (\ref{item: nef1}) hold for the movable cone \cite[Lemma 3.10]{FL}.
One can ask whether in general the two cones
$\overline{\Mov}_d(X)$ and $\overline{\WMov}_d(X)$ are the same.
This is true if and only if the cycle class of any nef subvariety lies in the movable cone.
This question seems closely related to Hartshorne's Conjecture A,
which was disproved by Fulton and Lazarsfeld \cite{FulLaz1}.
For any subvariety with ample normal bundle,
Hartshorne asked if some multiple of the underlying cycle class of the subvariety moves in a large algebraic family.

It is unclear what kind of intersection theoretic statements we should expect if we further assume that $Y$ has ample normal bundle.
Voisin gave an example of a subvariety with ample normal bundle such that its cycle class lies on the boundary of the pseudoeffective cone of cycles \cite{Voisin}.
On the other hand, Ottem showed that the cycle class of a curve with ample normal bundle lies in the interior of the cone of curves \cite{Ottem1}.
In a separate work, we studied the numerical dimension of a pseudoeffective divisor by restricting it to a subvariety with ample normal bundle \cite{Lau}.

It is interesting to note that the cone dual to the pseudoeffective cone of $d$-cycles is not in general pseudoeffective, this is a result by Debarre, Ein, Lazarsfeld and Voisin \cite{DELV}.
Ottem also found an example of a nef non-pseudoeffective
class on some hyperk\"ahler fourfold \cite{Ott15}.

All schemes in this work are over a field $k$ of characteristic $0$.

\begin{ack}
The author would like to thank his advisor, Tommaso de Fernex, for many hours of discussions on this project, as well as his kindness and support. 
He would also like to thank John Christian Ottem and Burt Totaro for their interests in this project.
He would like to sincerely thank both referees for their patience, careful proofreading,
suggestions that sharpen results,
interesting questions and even tips on typesetting. 
This is part of the author's PhD's thesis.
\end{ack}

\section{\texorpdfstring{$q$-}{q-}ample divisors and ample subschemes}
In this section, we shall first gather some useful facts about $q$-ample divisors, then we shall recall Ottem's definition of an ample subscheme and some of its properties.

Let us recall the definition of a $q$-ample line bundle.

\begin{definition}[$q$-ample line bundle {\cite{DPS},\cite{Totaro}}]
Let $X$ be a projective scheme and $q$ be a nonnegative integer. A line bundle $\mathscr{L}$ is \textit{$q$-ample} 
if for any coherent sheaf $\mathscr{F}$ on $X$, there is a positive integer $m_{0}$, depending on $\mathscr{L}$ and $\mathscr{F}$, such that
\[
H^{i}(X,\mathscr{F}\otimes\mathscr{L}^{\otimes m})=0
\]
for $i>q$ and $m>m_0$.
\end{definition}
\begin{lemma}[{\cite[Lemma 2.1]{Ottem}}]\label{lemma: Ottem}
Let $X$ be a projective scheme.
Fix an ample line bundle $\mathscr{O}(1)$ on $X$.
A line bundle $\mathscr{L}$ is $q$-ample if and only if for any $l\geq 0$,
\[
H^{i}\bigl(X,\mathscr{L}^{\otimes m}\otimes \mathscr{O}(-l)\bigr)=0
\]
for $m\gg 0$ and $i>q$.
\end{lemma}

We shall start with the definition of a Koszul-ample line bundle. 
The details are not very important in this paper,
but they are included for the sake of completeness.
One useful fact is that any large tensor power of an ample line bundle is $2n$-Koszul-ample, where $n$ is the dimension of the underlying projective scheme \cite{Backelin}.
\begin{definition}[Koszul-ampleness  {\cite[Section 1]{Totaro}}]
Let $X$ be a projective scheme of dimension $n$.
Suppose the ring of global regular function $\mathscr{O}(X)$ on $X$ is a field (e.g. $X$ is connected and reduced).
Given a very ample line bundle $\mathscr{O}_X(1)$, 
we say that it is \textit{$N$-Koszul ample} if the homogeneous coordinate ring $A=\bigoplus_{j} H^{0}(X,\mathscr{O}_X(j))$ is $N$-Koszul, i.e. there is a resolution
\[
\cdots\rightarrow M_1 \rightarrow M_0\rightarrow k\rightarrow 0
\]
where $M_i$ is a free $A$-module, generated in degree $i$, where $i\leq N$.
\end{definition}


\begin{definition}[$q$-T-ampleness {\cite[Definition 6.1]{Totaro}}]
Let $X$ be a projective scheme of dimension $n$.
Suppose the ring of regular functions of $X$, $\mathscr{O}(X)$ is a field. 
We fix a $2n$-Koszul-ample line bundle $\mathscr{O}_{X}(1)$ on $X$.
We say that a line bundle $\mathscr{L}$ is \textit{$q$-T-ample} if there is a positive integer $N$, such that
\[
H^{q+i}\bigl(X,\mathscr{L}^{\otimes N}(-n-i)\bigr)=0,
\]
for $1\leq i\leq n-q$.
\end{definition}

Totaro showed that $q$-T-ampleness is the same as $q$-ampleness \cite[Theorem 6.3]{Totaro}. 
Even though the notion of $q$-T-ampleness may appear technical, the equivalence is the key result of his paper.
It reduces the problem of showing a line bundle being $q$-ample to checking the vanishing of finitely many cohomology groups.
Using the notion of $q$-T-ampleness, Totaro showed that $q$-ampleness is Zariski open \cite[Theorem 8.1]{Totaro}.

The notion of $q$-ampleness can be extended to $\mathbf{R}$-Cartier $\mathbf{R}$-divisors.
\begin{definition}[$q$-ample $\mathbf{R}$-Cartier $\mathbf{R}$-divisors {\cite[Definition 2.20]{GK}}]
Let $X$ be a projective scheme.
An $\mathbf{R}$-Cartier $\mathbf{R}$-divisor $D$ on $X$ is \textit{$q$-ample} if $D$ is numerically equivalent to $cL+A$ with $L$ a $q$-ample line bundle, $c\in \mathbf{R}_{>0}$, $A$ an ample $\mathbf{R}$-Cartier $\mathbf{R}$-divisor.
\end{definition}
Based on the work of Demailly, Peternell and Schneider, Totaro also proved that 
\begin{theorem}[{\cite[Theorem 8.3]{Totaro}}]\label{thm: open}
An integral divisor is $q$-ample if and only if its associated line bundle is $q$-ample.
The set of $q$-ample $\mathbf{R}$-Cartier $\mathbf{R}$-divisors in $\N^{1}(X)_{\mathbf{R}}$ defines an open cone (but not convex in general) and that the sum of a $q$-ample $\mathbf{R}$-Cartier $\mathbf{R}$-divisor and an $r$-ample $\mathbf{R}$-Cartier $\mathbf{R}$-divisor is $(q+r)$-ample.
\end{theorem}
These facts are non-trivial.
We shall use the notion of $q$-T-ampleness to prove Proposition \ref{prop: pullback}.

We note that $(n-1)$-ampleness admits a pleasant geometric interpretation, which we shall use a few times in this paper.
\begin{theorem}[{\cite[Theorem 9.1]{Totaro}}]\label{theorem: n-1 ample}
Let $X$ be a projective variety of dimension $n$. 
A line bundle $\mathscr{L}$ on $X$ is $(n-1)$-ample if and only if
$[\mathscr{L}^{\vee}]\in \N^{1}(X)$ does not lie in the pseudoeffective cone.
\end{theorem}

We need the following result on the positivity of the pullback of a $q$-ample divisor.
\begin{proposition}[Pullback of a $q$-ample divisor]\label{prop: pullback}
Let $f:X'\rightarrow X$ be a morphism of projective schemes,
$D$ be a $q$-ample divisor on $X$, and $A$ be a relatively (to $f$) ample divisor on $X'$. 
Then $mf^{*}D+A$ is $q$-ample, for $m\gg0$.
\end{proposition}
\begin{proof}
First, let us show that it suffices to prove the proposition in the case where both $X$ and $X'$ are integral. 
Note that a line bundle is $q$-ample on $X'$ (resp. $f$-ample) if and only if its restriction to each irreducible component of $X'$ is $q$-ample \cite [Proposition 2.3.i,ii]{Ottem} (resp. $f$-ample). 
We may now assume that $X'$ is integral.
Let $X_{1}$ be an irreducible component of $X$ that contains the image of $X'$.
The morphism $f$ factors through $X_1$.
Since the restriction $D|_{X_{1}}$ is $q$-ample, we may replace $X$ by $X_1$.

Now we may assume both $X$ and $X'$ are integral.
In fact, we shall prove that $mf^{*}D+A$ is $q$-T-ample, for $m\gg 0$.
In other words, we shall show that for $m\gg 0$, there is a positive integer $r$, such that
\[
H^{q+a}\biggl(X', \mathscr{O}_{X'}\bigl(r(mf^{*}D+A)\bigr)\otimes \mathscr{O}_{X'}(-n-a)\biggr)=0
\]
for $1\leq a\leq n-q$. Here $\mathscr{O}_{X'}(1)$ is a $2n$-Koszul-ample line bundle on $X'$, where $n=\dim X'$.

Using the relative ampleness of $A$, 
one can find a positive integer $r$ such that 
\[
R^{j}f_{*}\bigl(\mathscr{O}_{X'}(rA)\otimes \mathscr{O}_{X'}(-n-a)\bigr) =0,
\]
for $j>0$ and $1\leq a\leq n-q$. 
The Leray spectral sequence then says
\begin{multline*}
H^{q+a}\biggl(X', \mathscr{O}_{X'}\bigl(r(mf^{*}D+A)\bigr)\otimes \mathscr{O}_{X'}(-n-a)\biggr)
\\
\cong
H^{q+a}\biggl(X,\mathscr{O}_{X}(rmD)\otimes
f_{*}\bigl(\mathscr{O}_{X'}(rA)\otimes
\mathscr{O}_{X'}(-n-a)\bigr)\biggr).
\end{multline*}
The group on the right hand side vanishes for all big $m$, 
by the $q$-ampleness of $rD$.
\end{proof}

We now review the definition of an ample subscheme, given by Ottem:
\begin{definition}[Ample subscheme {\cite[Definition 3.1]{Ottem}}]\label{definition: ample}
Let $X$ be a projective scheme,
$Y$ be a closed subscheme of $X$ of codimension $r$ and 
$\pi: \Bl_Y X\rightarrow X$ be the blowup of $X$ along $Y$.
We say that $Y$ is an \textit{ample subscheme} of $X$ if
the exceptional divisor $E$ of $\pi$ is $(r-1)$-ample in $\Bl_Y X$.
\end{definition}
We shall follow his definition in this paper. 
An example of an ample subscheme would be the zero locus (of codimension $r$) of a section of an ample vector bundle of rank $r$ \cite[Proposition 4.5]{Ottem}.
On the other hand, many good properties listed in Hartshorne's book \cite[p.XI]{Hartshorne} are satisfied under this definition.
Before stating some of these properties, we need the definition of \textit{cohomological dimension} of a scheme $U$:
It refers to the integer
    \[
    \cd(U):=\max\{i\in \mathbb{Z}_{\geq 0}|\, H^{i}(U,\mathscr{F})\neq 0 \textit{, for some coherent sheaf }\mathscr{F}.\}
    \]
\begin{theorem}\label{theorem: ample subvariety}
Let $Y$ be a smooth closed subscheme of a smooth projective scheme $X$. 
    \begin{enumerate}
        \item $Y$ is ample if and only if its normal bundle is ample and the cohomological dimension of the complement is $r-1$, i.e., $\cd(X\backslash Y)=r-1$. \label{theorem: normal bundle}
    \end{enumerate}
Assume further that $Y$ is an ample subscheme in $X$. Then
    \begin{enumerate}
    \setcounter{enumi}{1}
        \item Generalized Lefschetz hyperplane theorem with rational coefficients holds, i.e. 
        $H^{i}(X,\mathbb{Q})\rightarrow H^{i}(Y,\mathbb{Q})$ is an isomorphism for $i<\dim Y$ and is an injection for $i=\dim Y$. \label{theorem: lefschetz}
        \item $Y$ is numerically positive, i.e. $Y\cdot Z>0$ for any effective cycle $Z$ of dimension $r$. \label{theorem: intersection}
        \item \label{theorem: H^{i}}
        $H^{i}(X,\mathscr{F})\rightarrow H^{i}(\widehat{X},\widehat{\mathscr{F}})$ is an isomorphism for $i<\dim Y$ and is injective for $i=\dim Y$.  Here $\widehat{X}$ is the formal completion of $X$ along $Y$, $\mathscr{F}$ is a locally free sheaf on $X$ and $\widehat{\mathscr{F}}$ is its restriction to $\widehat{X}$. 
    \end{enumerate}
   
\end{theorem}
\begin{proof}
\cite[Theorem 5.4]{Ottem}, \cite[Corollary 5.3]{Ottem} and \cite[Chapter III, Theorem 3.4]{Hartshorne} give (\ref{theorem: normal bundle}), (\ref{theorem: lefschetz}) and (\ref{theorem: H^{i}}) respectively.
For (\ref{theorem: intersection}), since $\cd(X-Y)=r-1$, $Y$ meets any effective cycle of dimension $r$. We can then apply the result of Fulton and Lazarsfeld \cite[Corollary 8.4.3]{Laz}, which says if $Y$ has ample normal bundle and $Y$ meets $Z$, where $Z$ is an effective cycle of complementary dimension to that of $Y$, then $Y\cdot Z>0$. 
\end{proof}
The above list of properties is incomplete, for a more complete picture, c.f. \cite{Ottem}.

\section{Partial regularity and a Fujita-type vanishing theorem for \textit{q}-ample divisors}
We shall quickly go through the results in sections 2 and 3 in Totaro's paper \cite{Totaro}.
There Totaro developed on Arapura's idea \cite{Arapura} on using resolution of the diagonal to study the Castelnuovo-Mumford regularity of a coherent sheaf.
Using these ideas, we shall provide a weak extension of a vanishing theorem for $q$-ample line bundles proved by Totaro \cite[Theorem 6.4]{Totaro} (Theorem \ref{theorem: unif}).
From this, we prove a generalization of the Fujita vanishing theorem (Theorem \ref{theorem: fujita}) to the $q$-ample divisors setting.
It also generalizes the Fujita-type vanishing theorem that K\"uronya proved \cite[Theorem C]{Kuronya}.
One may also deduce Theorem \ref{theorem: fujita} from Keeler's generalization of Fujita vanishing theorem \cite[Theorem 4.5 and 5.5]{Keeler}, combined with Totaro's results on $q$-ample divisors.
We shall not pursue this here.
We shall later apply this theorem to prove Theorem \ref{theorem: restriction}, as well as Theorems \ref{theorem: trans ample} and \ref{theorem: trans nef}.

If we are only interested in the vanishing of the top cohomology,
we showed that the assumptions in Theorem \ref{theorem: fujita} can be weakened.
This is Proposition \ref{lemma: uniform2}.
It will be used later in the proof of Theorem \ref{theorem: pseudoeffective}.

In this section, we assume $X$ to be a projective scheme
of dimension $n$ over a field $k$,
with the ring of regular functions on $X$ being a field.
Furthermore, we fix a $2n$-Koszul-ample line bundle $\mathscr{O}_{X}(1)$ on $X$.

\begin{theorem} \label{thm: resolution of the diagonal}\cite[Theorem 2.1]{Totaro}
On $X\times_k X$, we have an exact sequence of coherent sheaves:
\begin{equation}\label{resolution}
\mathscr{R}_{2n-1}\boxtimes\mathscr{O}_{X}(-2n+1)\rightarrow
\cdots\rightarrow\mathscr{R}_1\boxtimes\mathscr{O}_X(-1)
\rightarrow\mathscr{R}_0\boxtimes\mathscr{O}_X
\rightarrow\mathscr{O}_{\Delta}
\rightarrow 0,
\end{equation}
where $\Delta\subset X\times_k X$ is the diagonal. 
Here all the $\mathscr{R}_i$'s are locally free sheaves on $X$ that can be fit into short exact sequences:
\begin{equation}\label{R}
0
\rightarrow \mathscr{R}_{i+1}\otimes \mathscr{O}_X(-1)
\rightarrow B_{i+1}\otimes_{k} \mathscr{O}_X(-1)
\rightarrow \mathscr{R}_{i} \rightarrow 0,
\end{equation}
where the $B_{i+1}$'s are $k$-vector spaces.
\end{theorem}

\begin{lemma}\label{lemma: vanishing of tensor product}\cite[Lemma 3.1]{Totaro}
Let $\mathscr{E}$ and $\mathscr{F}$ be a locally free sheaf and a coherent sheaf on $X$ respectively. 
Suppose for each pair of integers $(a,b)$, such that $0\leq a\leq 2n-i$ and $b\geq 0$,
$H^{b}(\mathscr{E}\otimes \mathscr{R}_a)=0$ or
$H^{i+a-b}\bigl(\mathscr{F}(-a)\bigr)=0$.
Then
$H^{i}(\mathscr{E}\otimes\mathscr{F})=0$.
\end{lemma}
\begin{hproof}
After tensoring with $\mathscr{E}\boxtimes\mathscr{F}$, the sequence (\ref{resolution}) remains exact, we now apply K\"unneth's formula.
\end{hproof}

\begin{definition}[Partial regularity {\cite[Definition 2.1]{Keeler}}]\label{def: reg}
We fix a $2n$-Koszul ample line bundle $\mathscr{O}_{X}(1)$.
Let $\mathscr{G}$ be a coherent sheaf on $X$ and let $q$ be a non-negative integer.
We say that $\mathscr{G}$ is \textit{$(0,q)$-regular} if
\begin{equation}\label{q-regular}
H^{q+i}\bigl(X,\mathscr{G}\otimes \mathscr{O}_{X}(-i)\bigr)= 0
\end{equation}
for all $\ 1\leq i \leq n-q$.

We set
\[
\reg^{q}(\mathscr{F})= \inf\{m\in \mathbb{Z}\,| \, \mathscr{F}\otimes\mathscr{O}_{X}(m) \text{ is } (0,q) \text{-regular.}\}
\]
\end{definition}
When $q=0$, this is just the usual Castelnuovo-Mumford regularity of the coherent sheaf $\mathscr{F}$, relative to $\mathscr{O}_{X}(1)$. 
It is clear that $\reg^{q}(\mathscr{F})\in[-\infty,+\infty)$, by the ampleness of $\mathscr{O}_{X}(1)$.
\begin{lemma}[{\cite[Lemma 2.2]{Keeler}}]\label{lemma: regularity}
If $\mathscr{F}$ is $(0,q)$-regular, then $\mathscr{F}\otimes\mathscr{O}_{X}(1)$ is also $(0,q)$-regular.
\end{lemma}

\begin{lemma}[{\cite[Lemma 3.3]{Totaro}}]\label{lemma: R_i}
If $\mathscr{F}$ is a $(0,q)$-regular coherent sheaf on $X$, then
\[
H^{j}\bigl(X,\mathscr{F}\otimes \mathscr{R}_{i}\bigr)=0
\]
for $j>q$ and $i<n+j$. Here, we are referring to the $\mathscr{R}_{i}$'s that appear in Theorem \ref{thm: resolution of the diagonal}.
\end{lemma}

Next, we generalize \cite[Theorem 3.4]{Totaro}.

\begin{theorem}[Subadditivity of partial regularity]\label{theorem: subadditivity}
Let $\mathscr{E}$ and $\mathscr{F}$ be a locally free sheaf and a coherent sheaf on $X$ respectively, then 
\[
\reg^{q}(\mathscr{E}\otimes \mathscr{F})\leq \reg^{l}(\mathscr{E}) + \reg^{q-l}(\mathscr{F})
\]
for any $0\leq l\leq q$.
\end{theorem}
\begin{proof}
After replacing $\mathscr{E}$ and $\mathscr{F}$ by $\mathscr{E}\otimes\mathscr{O}_{X}(\reg^l(\mathscr{E}))$ and $\mathscr{F}\otimes\mathscr{O}_{X}(\reg^{q-l}(\mathscr{F}))$ respectively, 
we may assume $\mathscr{E}$ and $\mathscr{F}$ are $(0,l)$-regular and $(0,q-l)$-regular, respectively.
It suffices to show that $\mathscr{E}\otimes \mathscr{F}$ is $(0,q)$-regular, i.e.,
\[
H^{q+i}\bigl(X,\mathscr{E}\otimes\mathscr{F}\otimes\mathscr{O}_X(-i)\bigr)=0
\]
for $1\leq i\leq n-q$.
To this end, we shall check that 
for each pair of integers $(a,b)$ such that $0\leq a\leq 2n-(q+i)$ and $b\geq 0$,
\[
H^{b}(X,\mathscr{E}\otimes \mathscr{R}_a)=0
\]
or
\[
H^{(q+i)+a-b}\bigl(X,\mathscr{F}(-a-i)\bigr)=0
\]
and conclude by applying Lemma \ref{lemma: vanishing of tensor product}.

\begin{case} 
$b>l$ and $a<n+b$.
\end{case}
By Lemma \ref{lemma: R_i},
\[
H^{b}(X,\mathscr{E}\otimes \mathscr{R}_{a})=0.
\]

\begin{case}
$b>l$ and $n+b\leq a\leq 2n-(q+i)$.
\end{case}
Since $q+i+a-b\geq q+i+n>n=\dim X$,
\[
H^{(q+i)+a-b}\bigl(X,\mathscr{F}\otimes\mathscr{O}_{X}(-a-i)\bigr)=0.
\]

\begin{case}
$0\leq b\leq l$ and $0\leq a \leq 2n-(q+i)$.
\end{case}
By $(0,q-l)$-regularity of $\mathscr{F}$,
Lemma \ref{lemma: regularity} and the fact that $q-b\geq q-l$,
\[
H^{(q-b)+a+i}\bigl(X,\mathscr{F}\otimes\mathscr{O}_{X}(-a-i)\bigr)=0.
\]
This proves the theorem.
\end{proof}

We next prove an analogue of \cite[Theorem 6.4]{Totaro}.
\begin{theorem}[Uniform vanishing]\label{theorem: unif}
Let $\mathscr{L}$ be a $q$-ample line bundle on $X$. 
Then for any integer $N$, there is an integer $m_{N}$, 
such that, for any coherent sheaf $\mathscr{F}$ on $X$ with $\reg^{q'}(\mathscr{F})\leq N$,
\[
H^{i}(X,\mathscr{F}\otimes\mathscr{L}^{\otimes m})=0
\]
for $i>q+q'$ and $m>m_{N}$.
\end{theorem}
\begin{proof}
Fix an integer $i$ such that $q+q'< i \leq n$.
By Lemma \ref{lemma: vanishing of tensor product}, it is enough to show that there is a positive integer $M$, depending only on the choice of $N$, but not the coherent sheaf $\mathscr{F}$, such that for $m>M$, 
$0\leq a \leq 2n-i$ and $b\geq 0$,
\[
H^{b}(X,\mathscr{L}^{\otimes m}\otimes \mathscr{O}_{X}(-N)\otimes \mathscr{R}_{a})=0
\]
or 
\[
H^{i+a-b}\bigl(X,\mathscr{F}\otimes\mathscr{O}_{X}(N-a)\bigr).
\]

\setcounter{case}{0}
\begin{case}
$b> q$ and $0\leq a < n+b$.
\end{case}
Using the $q$-ampleness of $\mathscr{L}$, there is an $m_N$, such that we have
\[
H^{q+j}\bigl(X,\mathscr{L}^{\otimes m}\otimes \mathscr{O}_{X}(-N-j)\bigr)=0
\]
for all $1\leq j\leq n-q$ and $m>m_N$,
i.e. $\mathscr{L}^{\otimes m}\otimes \mathscr{O}_X(-N)$ is $(0,q)$-regular for all $m>m_N$. 
Now Lemma \ref{lemma: R_i} says
\[
H^{b}\bigl(X,\mathscr{L}^{\otimes m}\otimes \mathscr{O}_{X}(-N)\otimes \mathscr{R}_{a}\bigr)=0
\]
for all $m>m_N$, $b> q$ and $a< n+b$.

\begin{case}
$b>q$ and $n+b\leq a \leq 2n-i$.
\end{case}
Since $i+a-b\geq i+n > n=\dim X$,
\[
H^{i+a-b}\bigl(X,\mathscr{F}\otimes\mathscr{O}_{X}(N-a)\bigr)=0.
\]

\begin{case}
$0\leq b\leq q$ and $0\leq a\leq 2n-i$.
\end{case}
By the partial regularity assumption of $\mathscr{F}$,
the fact that $i-b>q'$
and Lemma \ref{lemma: regularity},
\[
H^{(i-b)+a}\bigl(X,\mathscr{F}\otimes\mathscr{O}_{X}(N-a)\bigr)=0.
\]

This proves the theorem.
\end{proof}

\begin{lemma}\label{lemma: reg of nef}
There is an $N$ such that $\reg^{0}(\mathscr{P})\leq N$ for any nef line bundle $\mathscr{P}$ on $X$.
\end{lemma}
\begin{proof}
By the Fujita vanishing theorem, there is an $N$ such that
\[
H^{a}\bigl(X,\mathscr{O}_X(N-a)\otimes \mathscr{P}\bigr)=0
\]
for $a>0$ and any nef line bundle $\mathscr{P}$.
\end{proof}

We may now prove a Fujita-type vanishing theorem for $q$-ample divisors.
Note that we do not assume $\mathscr{O}(Z)$ is a field in the following.
\begin{theorem}[Fujita-type vanishing theorem for $q$-ample divisors]\label{theorem: fujita}
Let $Z$ be a projective scheme of dimension $n$ over $k$,
$\mathscr{L}_j$ be $q_j$-ample line bundles on $Z$, 
$1\leq j\leq k$ and $\mathscr{F}$ be a coherent sheaf on $Z$. 
Then for any $(k-1)$-tuple $(M_2,\cdots,M_k)\in \mathbb{Z}^{k-1}$, there is a positive integer $M_1$, such that
\[
H^{i}(Z,\mathscr{F}\otimes\mathscr{L}_1^{\otimes m_1}\otimes \mathscr{L}_2^{\otimes m_2}\otimes\cdots\otimes \mathscr{L}_k^{\otimes m_k}\otimes \mathscr{P})=0
\]
for $i>\sum_{j=1}^k q_j$, $m_j\geq M_j$, where $1\leq j\leq k$, and any nef line bundle $\mathscr{P}$ on $Z$.
\end{theorem}
\begin{proof}
Clearly, we may assume that $Z$ is connected.
It suffices to prove the theorem assuming that $Z$ is also reduced.
To see this, let $\mathscr{N}$ be the nilradical ideal sheaf of $Z$.
We then chase through the following exact sequence:
\begin{multline*}
0\rightarrow
\mathscr{N}^{e+1}\cdot\mathscr{F}\otimes\mathscr{L}_1^{\otimes m_1}\otimes \mathscr{L}_2^{\otimes m_2}\otimes\cdots\otimes \mathscr{L}_k^{\otimes m_k}\otimes \mathscr{P}\\
\rightarrow
\mathscr{N}^{e}\cdot\mathscr{F}\otimes\mathscr{L}_1^{\otimes m_1}\otimes \mathscr{L}_2^{\otimes m_2}\otimes\cdots\otimes \mathscr{L}_k^{\otimes m_k}\otimes \mathscr{P}\\
\rightarrow
(\mathscr{N}^{e}\cdot\mathscr{F}/\mathscr{N}^{e+1}\cdot\mathscr{F})\otimes\mathscr{L}_1^{\otimes m_1}\otimes \mathscr{L}_2^{\otimes m_2}\otimes\cdots\otimes \mathscr{L}_k^{\otimes m_k}\otimes \mathscr{P}
\rightarrow 0.
\end{multline*}
Note that $(\mathscr{N}^{e}\cdot\mathscr{F}/\mathscr{N}^{e+1}\cdot\mathscr{F})\otimes\mathscr{L}_1^{\otimes m_1}\otimes \mathscr{L}_2^{\otimes m_2}\otimes\cdots\otimes \mathscr{L}_k^{\otimes m_k}\otimes \mathscr{P}$ is a coherent sheaf on $Z_{red}$, which is the reduction of $Z$,
and that $\mathscr{N}^{e}=0$ for $e\gg0$.

Since we now assume that $Z$ is connected and reduced, $\mathscr{O}(Z)$ is a field.
We may apply the previous results in this section on $Z$.
We fix a $2n$-Koszul-ample line bundle $\mathscr{O}_Z(1)$ on $Z$ and by $\reg^q$, we refer to partial regularity with respect to $\mathscr{O}_Z(1)$.

Since $\mathscr{L}_j$ is $q_j$-ample, 
\[
H^{q_j+a}\bigl(Z,\mathscr{L}_j^{\otimes m_j}\otimes\mathscr{O}(-a)\bigr)=0
\]
for $m_j\gg0$ and $1\leq a \leq n-q_j$.
This says $\reg^{q_j}(\mathscr{L}_j^{\otimes m_j})\leq 0$ for all $m_j\gg0$.
Therefore, there exist $N_j$'s such that 
$\reg^{q_j}(\mathscr{L}_j^{\otimes m_j})\leq N_j$ for all $m_j\geq M_j$.
We now apply Theorem \ref{theorem: subadditivity} 
and Lemma \ref{lemma: reg of nef}
to see that
\[
\reg^{\sum_{j=2}^k q_j}(\mathscr{F}\otimes \mathscr{L}_2^{\otimes m_2}\otimes\cdots\otimes \mathscr{L}_k^{\otimes m_k}\otimes \mathscr{P})\leq \reg^0(\mathscr{F})+\sum_{j=2}^k N_j + N
\]
for all $m_j\geq M_j$, where $2\leq j\leq k$ and $N$ is the bound mentioned in Lemma \ref{lemma: reg of nef}.
Finally, we apply Theorem \ref{theorem: unif} to get the desired result. 
\end{proof}

Let us introduce the notion of $q$-almost ample divisors which generalizes the notion of nef divisors.

\begin{definition}[$q$-almost ample]
Let $X$ be a projective scheme, $D$ an $\mathbf{R}$-Cartier $\mathbf{R}$-divisor on $X$ and $A$ an ample divisor on $X$.
We say that $D$ is \textit{$q$-almost ample} if $D+\epsilon A$ is $q$-ample for $0<\epsilon\ll 1$.
\end{definition}

The definition is clearly independent of the choice of $A$.
Note that $D$ is $0$-almost ample if and only if $D$ is nef.

In the following proposition, we only look at the vanishing of the top cohomology.
Note that we allow all but one of the line bundles in question to be $q_i$-almost ample (Compare with Theorem \ref{theorem: fujita}).

\begin{proposition}\label{lemma: uniform2}
Let $Z$ be a projective scheme of dimension $n$ over $k$,
$\mathscr{L}$ be a $q$-ample line bundle on $Z$ and $\mathscr{F}$ be a coherent sheaf on $Z$.
Then there is a positive integer $M$,
such that given any finite set of line bundles $\{\mathscr{P}_i\}_{1\leq i\leq k}$,
with each line bundle being $q_i$-almost ample and $q+\sum_i q_i\leq n-1$,
\[
H^{n}(Z,\mathscr{F}\otimes\mathscr{L}^{\otimes m}\otimes \bigotimes_{i=1}^{k}\mathscr{P}_i)=0
\]
for $m\geq M$.
\end{proposition}
\begin{proof}
Let us first reduce to the case where $Z$ is integral.
Indeed, arguing as in the proof of Theorem \ref{theorem: fujita}, we may assume $Z$ is reduced.
Suppose $Z=\bigcup_{i=1}^{k}Z_i$, where $Z_i$ are the irreducible components of $Z$.
Let $\mathscr{I}$ be the ideal sheaf of $Z_1\subset Z$.
Consider the short exact sequence
\[
0\rightarrow
\mathscr{I}\cdot \mathscr{F}
\rightarrow
\mathscr{F}
\rightarrow
\mathscr{F}/\mathscr{I}\cdot \mathscr{F}
\rightarrow 0.
\]
Note that $\mathscr{I}\cdot \mathscr{F}$ and $\mathscr{F}/\mathscr{I}\cdot \mathscr{F}$ are supported on $\bigcup_{i=2}^k Z_i$ and $Z_1$ respectively.
We then tensor the above short exact sequence with $\mathscr{L}^{\otimes m}\otimes \bigotimes_{i=1}^{k}\mathscr{P}_i$
and induct on the number of irreducible components of $Z$.
Therefore, we may assume that $Z$ is irreducible as well.

Fix an ample line bundle $\mathscr{O}_Z(1)$ on $Z$.
We can find a surjection
$\oplus\mathscr{O}_Z(a)\twoheadrightarrow\mathscr{F}$, for some $a\in\mathbb{Z}$.
Thus, it suffices to only consider the case where $\mathscr{F}$ is a line bundle $\mathscr{M}$.
Let $\omega_Z$ be the dualizing sheaf of $Z$ \cite[III.7]{AG}.
We have
\[
H^{n}(Z,\mathscr{M}\otimes\mathscr{L}^{\otimes m}\otimes \bigotimes_{i=1}^{k}\mathscr{P}_i)
\cong
H^{0}(Z,\mathscr{M}^{\vee}\otimes\mathscr{L}^{\otimes -m}\otimes \bigotimes_{i=1}^{k}\mathscr{P}^{\vee}_i\otimes\omega_Z)^{\vee}.
\]
We can embed $\omega_Z\hookrightarrow\mathscr{O}(j)$ \cite[Proof of Theorem 9.1]{Totaro}.
This reduces us to proving the vanishing of the cohomology group
$H^{0}(Z,\mathscr{M}^{\vee}\otimes\mathscr{O}(j)\otimes\mathscr{L}^{\otimes -m}\otimes \bigotimes_{i=1}^{k}\mathscr{P}^{\vee}_i)$.
We may find a positive integer $M$ such that $\mathscr{L}^{\otimes m}\otimes\mathscr{M}\otimes\mathscr{O}(-j)$ is $q$-ample for all $m\geq M$, by the openness property of the $q$-ample cone (Theorem \ref{thm: open}).
By the convexity property (Theorem \ref{thm: open}),
\[
\bigl(\mathscr{L}^{\otimes m}\otimes\mathscr{M}\otimes\mathscr{O}(-j)\bigr)
\otimes
\bigotimes_{i=1}^{k}\mathscr{P}_i
\textit{ is }
q+\sum_i q_i
\textit{-ample},
\]
hence $(n-1)$-ample, for $m\geq M$.
By Theorem \ref{theorem: n-1 ample}, 
\[
\mathscr{L}^{\otimes -m}\otimes\mathscr{M}^{\vee}\otimes\mathscr{O}(j)
\otimes
\bigotimes_{i=1}^{k}\mathscr{P}^{\vee}_i
=
\bigl(\mathscr{L}^{\otimes m}\otimes\mathscr{M}\otimes\mathscr{O}(-j)
\otimes
\bigotimes_{i=1}^{k}\mathscr{P}_i)^{\vee}
\]
is not pseudoeffective for $m\geq M$.
Therefore, it cannot have any global sections.
\end{proof}

\section{Nef subschemes}
In this section, we shall define the notion of nef subschemes.
We shall show that ampleness and nefness are transitive properties:
If $Z$ is an ample (resp. nef) subscheme of $Y$ and $Y$ is an ample (resp. nef) subscheme of $X$,
then $Z$ is an ample (resp. nef) subscheme of $X$ (Theorems \ref{theorem: trans ample} and \ref{theorem: trans nef}).
We shall study nef subvarieties more in sections \ref{s6} and \ref{s7}.

Ottem observed that ampleness of a vector bundle $\mathscr{E}$ can be expressed in terms of partial ampleness of $\mathscr{O}_{\mathbb{P}(\mathscr{E}^{\vee})}(-1)$ \cite[Proposition 4.1]{Ottem}.
We give the straightforward generalization to the case where the vector bundle is nef.
\begin{proposition}\label{prop: vector bundle}
Let $\mathscr{E}$ be a vector bundle of rank $r$ on a projective scheme $X$.
Then $\mathscr{E}$ is ample (resp. nef) if and only if $\mathscr{O}_{\mathbb{P}(\mathscr{E}^{\vee})}(-1)$ is $(r-1)$-ample.
(resp. $(r-1)$-almost ample.)
\end{proposition}
\begin{proof}
Let $\pi': \mathbb{P}(\mathscr{E}^{\vee})\rightarrow X$ and $\pi: \mathbb{P}(\mathscr{E})\rightarrow X$ be the natural projections.
Using \cite[Exercise III.8.4]{AG}, we have for $m>0$,
\[
R^{j}\pi'_{*}\mathscr{O}_{\mathbb{P}(\mathscr{E}^{\vee})}(-m-r)
\cong
\left\{
\begin{array}{ll}
\Sym^{m}\mathscr{E}\otimes\det(\mathscr{E}) & \text{ for } j=r-1;\\
0 & \text{ otherwise.}
\end{array}
\right.
%
\]
Here we implicitly used the fact that 
$(\Sym^{m}\mathscr{E}^{\vee})^{\vee}$ and $\Sym^{m}\mathscr{E}$ are isomorphic when the ground field is of characteristic $0$.

Therefore, we have the following isomorphisms
\begin{multline}\label{eq: vb}
H^{r-1+i}\bigl(\mathbb{P}(\mathscr{E}^{\vee}),\mathscr{O}_{\mathbb{P}(\mathscr{E}^{\vee})}(-m-r)\otimes \pi{'}^{*}(\mathscr{F}\otimes \det \mathscr{E}^{\vee})\bigr)
\cong
H^{i}(X,\Sym^{m}\mathscr{E}\otimes \mathscr{F})\\
\cong
H^{i}(\mathbb{P}(\mathscr{E}),\mathscr{O}_{\mathbb{P}(\mathscr{E})}(m)\otimes\pi^{*}\mathscr{F}),
\end{multline}
where $\mathscr{F}$ is locally free on $X$, $i>0$ and $m>0$.

Throughout the proof, in order to apply Lemma \ref{lemma: Ottem}, we fix a non-negative integer $l$. 
Choose a sufficiently ample divisor $A$ on $X$,
such that $\mathscr{O}_{\mathbb{P}(\mathscr{E})}(1)\otimes\pi^{*}\mathscr{O}(A)$ 
and $\mathscr{O}_{\mathbb{P}(\mathscr{E}^{\vee})}(1)\otimes\pi'^{*}\mathscr{O}(A)$ are ample.

If $\mathscr{O}_{\mathbb{P}(\mathscr{E}^{\vee})}(-1)$ is $(r-1)$-ample, 
then $\mathscr{O}_{\mathbb{P}(\mathscr{E})}(1)$ is ample, by (\ref{eq: vb}). 
Indeed, any line bundle on $\mathbb{P}(\mathscr{E})$ can be expressed as $\pi^{*}\mathscr{L}\otimes\mathscr{O}_{\mathbb{P}(\mathscr{E})}(a)$, for some line bundle $\mathscr{L}$ on $X$ and $a\in\mathbb{Z}$.

Suppose $\mathscr{O}_{\mathbb{P}(\mathscr{E}^{\vee})}(-1)$ is $(r-1)$-almost ample.
Fix a positive integer $k$.
Observe that we have the following isomorphism given by (\ref{eq: vb}):
\begin{multline*}
H^{i}\biggl(\mathbb{P}(\mathscr{E}), \mathscr{O}_{\mathbb{P}(\mathscr{E})}(mk-l)\otimes\pi^{*}\mathscr{O}\bigl((m-l)A\bigr)\biggr)\\
\cong
H^{r-1+i}\biggl(\mathbb{P}(\mathscr{E}^{\vee}),\mathscr{O}_{\mathbb{P}(\mathscr{E}^{\vee})}(-mk-r+l)\otimes \pi{'}^{*}\bigl(\mathscr{O}((m-l)A)\otimes \det \mathscr{E}^{\vee}\bigr)\biggr),
\end{multline*}
where $i,m>0$.
The latter term vanishes for $m\gg 0$ since
$\mathscr{O}_{\mathbb{P}(\mathscr{E}^{\vee})}(-k)\otimes\pi{'}^{*}\mathscr{O}(A)$ is $(r-1)$-ample.
This shows that
$\mathscr{O}_{\mathbb{P}(\mathscr{E})}(k)\otimes \pi{'}^{*}\mathscr{O}(A)$ is ample.
Since $\mathscr{O}_{\mathbb{P}(\mathscr{E})}(k)\otimes \pi{'}^{*}\mathscr{O}(A)$ is ample for any positive integer $k$, $\mathscr{O}_{\mathbb{P}(\mathscr{E})}(1)$ is nef.


Now assume that $\mathscr{E}$ is ample.
By (\ref{eq: vb}), we have the following isomorphism of cohomology groups:
\begin{multline*}
H^{r-1+i}\bigl(\mathbb{P}(\mathscr{E}^{\vee}),\mathscr{O}_{\mathbb{P}(\mathscr{E}^{\vee})}(-m-l)\otimes \pi{'}^{*}\mathscr{O}(-lA)\bigr)\\
\cong
H^{i}\biggl(\mathbb{P}(\mathscr{E}), \mathscr{O}_{\mathbb{P}(\mathscr{E})}(m+l-r)\otimes\pi^{*}\bigl(\mathscr{O}(-lA)\otimes\det\mathscr{E}\bigr)\biggr).
\end{multline*}
The latter term vanishes for $i>0$ and $m\gg0$, which implies that $\mathscr{O}_{\mathbb{P}(\mathscr{E}^{\vee})}(-1)$ is $(r-1)$-ample.

Assume that $\mathscr{E}$ is nef.
Fix a positive integer $k$.
We have the following isomorphism of cohomology groups:
\begin{multline*}
H^{r-1+i}\biggl(\mathbb{P}(\mathscr{E}^{\vee}),\mathscr{O}_{\mathbb{P}(\mathscr{E}^{\vee})}(-mk-l)\otimes \pi{'}^{*}\mathscr{O}\bigl((m-l)A\bigr)\biggr)\\
\cong
H^{i}\biggl(\mathbb{P}(\mathscr{E}), \mathscr{O}_{\mathbb{P}(\mathscr{E})}(mk+l-r)\otimes\pi^{*}\bigl(\mathscr{O}((m-l)A)\otimes\det\mathscr{E}\bigr)\biggr),
\end{multline*}
for $i,m>0$. The latter term vanishes for $m\gg0$. 
This implies that $\mathscr{O}_{\mathbb{P}(\mathscr{E}^{\vee})}(-k)\otimes\mathscr{O}(A)$ is $(r-1)$-ample, for any positive integer $k$.
It follows that $\mathscr{O}_{\mathbb{P}(\mathscr{E}^{\vee})}(-1)$ is $(r-1)$-almost ample.
\end{proof}

The augmented base locus gives us another measure of how far a divisor is from being ample. 
\begin{definition}[Augmented base locus {\cite[Definition 1.2]{ELMNP}}]\label{def: aug}
The augmented base locus of an $\mathbf{R}$-Cartier $\mathbf{R}$-divisor D on
X is the Zariski-closed subset:
\[
\mathbf{B}_{+}(D)=\bigcap_{D=A+E} \Supp E,
\]
where the intersection is taken over all decompositions $D=A+E$ with $A$ being ample and $E$ being effective.
\end{definition}

\begin{proposition}
\label{prop: exceptional}
Let $Y\subset X$ be a subscheme of codimension $r$.
Then the normal bundle of the exceptional divisor $E$ in $\Bl_Y X$, $\mathscr{O}_E(E)$ is $(r-1)$-almost ample
if and only if
$E\subset \Bl_Y X$ is $(r-1)$-almost ample.
\end{proposition}

\begin{proof}
The "if" part of the statment is clear, since restriction of a $q$-ample divisor to a subscheme is always $q$-ample.
For the "only if" part, observe that $\mathbf{B}_{+}(E+\epsilon A)\subseteq \Supp E$ for $0<\epsilon\ll1$, where $A$ is an ample divisor on $X$.
Recall that Brown's theorem \cite[Theorem 1.1]{Brown} says that an $\mathbf{R}$-Cartier $\mathbf{R}$-divisor $D$ is $q$-ample if and only if $D|_{\mathbf{B}_{+}(D)}$ is $q$-ample.
We conclude by applying his theorem to $E+\epsilon A$.
\end{proof}

\begin{definition}[Nef subscheme]
Let $Y$ be a closed subscheme of codimension $r$ of a projective scheme $X$, 
and $E$ be the exceptional divisor in $\Bl_Y X$.
We define $Y$ to be a \textit{nef subscheme of $X$} if
$\mathscr{O}_E(E)$ is $(r-1)$-almost ample.
\end{definition}

\begin{remark}
Proposition \ref{prop: exceptional} says that $Y$ is a nef subscheme if and only if $E$ is $(r-1)$-almost ample in $\Bl_Y X$. 
If $Y$ is l.c.i. in $X$, then $Y$ is nef if and only if the normal bundle $\mathscr{N}_{Y/X}$ is nef (Proposition \ref{prop: vector bundle}). 
The advantage of making this more general definition, without requiring $Y$ to be l.c.i., is to include more subschemes that are apparently "positive", for example, a closed point that lies in the singular locus of $X$, or if $Y$ is a smooth subvariety with nef normal bundle, the subscheme of $X$ defined by a power of the ideal sheaf of $Y$ is also considered as nef in this definition.
\end{remark}

The zero locus of a section of an ample vector bundle is an ample subscheme, provided that the codimension is maximal \cite[Proposition 4.5]{Ottem}.
The analogous statement holds for the zero locus of a section of a nef vector bundle. The proof is essentially the same as in the ample case \cite[Proof of Proposition 4.5]{Ottem}.
\begin{proposition}
Let $\mathscr{E}$ be a nef vector bundle on $X$ of rank $r\leq n$ and $Y$ be the zero locus of a global section of $\mathscr{E}$.
If the codimension of $Y$ is $r$, then $Y$ is a nef subscheme.
\end{proposition}

The following proposition is a direct generalization of \cite[Proposition 3.4]{Ottem}.
\begin{proposition}[Equidimensionality of nef subschemes]\label{prop: equi}
Suppose $Y$ is a nef subscheme of a projective scheme $X$.
Then the restriction of the blowup morphism of $X$ along $Y$, to the exceptional divisor $E$, $\pi|_{E}:E\rightarrow Y$, is equidimensional.
In particular, $Y$ is pure dimensional.
\end{proposition}
\begin{proof}
Suppose $Y\subset X$ has codimension $r$.
Let $y\in Y$ be a closed point, we want to show that the fiber over $y$, $Z:=\pi^{-1}(y)$, is of dimension $(r-1)$.
Note that $E$ has dimension $n-1$, where $n=\dim X$.
This implies that $\dim Z\geq r-1$.

On the other hand, $-E$ is $\pi$-ample.
Therefore, $(-E-\epsilon A)|_Z$ is ample for $1\gg \epsilon >0$, where $A$ is an ample divisor on $E$.
By nefness of $Y$,
$\mathscr{O}_{E}(E+\epsilon A)$ is $(r-1)$-ample, for $1\gg \epsilon >0$. 
It follows from the convexity property (Theorem \ref{thm: open}) that $Z$ has dimension at most $(r-1)$.
\end{proof}

\begin{proposition}[Inverse image of nef subschemes]\label{prop: pullback of nef}
Suppose $Y$ is a nef subscheme of $X$ of codimension $r$, $p: X'\rightarrow X$ a morphism from an equidimensional projective scheme $X'$. If $p^{-1}(Y)$ has codimension $r$ in $X'$, then $p^{-1}(Y)$ is nef in $X'$. In particular, if $p$ is equidimensional, then $p^{-1}(Y)$ is nef.
\end{proposition}
\begin{proof}
We have the following commutative diagram
\cite[Corollary II.7.15]{AG}:

\[
\xymatrix{
\Bl_{p^{-1}(Y)}(X') \ar[r]^{\widetilde{p}} \ar[d] & \Bl_Y(X) \ar[d] \\
X' \ar[r]_{p} & X,
}
\]
with $\tilde{p}$ induced by the universal property of blowup and $\tilde{p}^{*}(E)=E'$, where $E$ and $E'$ are the exceptional divisors in the respective blowups. 
We now apply Proposition \ref{prop: pullback} to conclude the proof. 
\end{proof}

\begin{remark}
It is necessary to assume that the codimension of $Y$ is unchanged after taking inverse image.
Consider the case where $Y$ is a closed point of $X$ and $p$ is the blowup morphism of $X$ along $Y$.
Note that $Y$ is nef, even ample, in $X$, but the exceptional divisor is not nef.
\end{remark}
\begin{proposition}\label{prop: ample intersect}
Let $Y$ be an ample (resp. nef) subscheme of codimension $r$ of $X$.
Let $Z$ be a closed subscheme of $X$.
If $Y\cap Z$ has codimension $r$ in $Z$,
then $Y\cap Z$ is an ample (resp. nef) subscheme of $Z$.
\end{proposition}
\begin{proof}
Indeed, we have the following commutative diagram \cite[Corollary II.7.15]{AG}:
\[
\xymatrix{
\Bl_{Y\cap Z} Z \ar@{^{(}->}[r] \ar[d]_{\pi_Z}
& \Bl_{Y}X \ar[d]^{\pi_X}\\
Z \ar@{^{(}->}[r] & X.
}
\]

Note that the exceptional divisor of $\pi_Z$ is the restriction of the exceptional divisor $E$ of $\pi_X$.
If $E$ is $(r-1)$-ample (resp. $(r-1)$-almost ample), so is $E|_{\Bl_{Y\cap Z}Z}$. 
\end{proof}

Theorem \ref{theorem: trans ample} generalizes the transitivity property of ample subschemes \cite[Proposition 6.4]{Ottem} in the sense that we do not require $Y$ (resp. $Z$) to be l.c.i. in $X$ (resp. $Y$).
This gives further evidence that Ottem's definition of an ample subscheme is a natural one.
First, we need a couple of lemmata:
\begin{lemma}\label{lemma: pass to blowup}
Let $X$ be a projective scheme and
$Y$ be a closed subscheme of $X$ of codimension $r$.
Fix an ample line bundle $\mathscr{O}_X(1)$ on $X$.
Suppose the blowup of $X$ along $Y$, $\pi:\Bl_Y X\rightarrow X$, has fiber dimension at most $r-1$,
and $\mathscr{L}$ is a line bundle on $X$
such that its restriction to $Y$, $\mathscr{L}|_Y$, is $q$-ample.
If for any $l\geq 0$,
\[
H^{i}\biggl(\Bl_Y X, \pi^{*}\bigl(\mathscr{L}^{\otimes m}\otimes\mathscr{O}_X(-l)\bigr)\biggr)=0
\]
for $i>q+r$ and $m\gg 0$,
then $\mathscr{L}$ is $(q+r)$-ample.
\end{lemma}
\begin{proof}
Applying the Leray spectral sequence, we have
\[
E^{p,s}_{2}=H^{p}\bigl(X, R^{s}\pi_{*}\mathscr{O}_{\Bl_Y X}\otimes \mathscr{L}^{\otimes m}\otimes\mathscr{O}_X(-l)\bigr)
\Rightarrow H^{p+s}\biggl(\Bl_Y X, \pi^{*}\bigl(\mathscr{L}^{\otimes m}\otimes\mathscr{O}_X(-l)\bigr)\biggr).
\]

Since the fiber dimension of $\pi$ is at most $r-1$,
$R^{s}\pi_{*}\mathscr{O}_{\Bl_Y X}$ and 
$E^{p,s}_{2}=0$
for $s>r-1$
\cite[Corollary III.11.2]{AG}.

For $s>0$, $R^{s}\pi_{*}\mathscr{O}_{\Bl_Y X}$ is a coherent sheaf \underline{on $Y$}. 
Indeed, this follows by considering the long exact sequence
\[
\cdots\rightarrow R^{s}\pi_{*}\mathscr{O}_{\Bl_Y X}(-jE)
\rightarrow R^{s}\pi_{*}\mathscr{O}_{\Bl_Y X}\bigl((-j+1)E\bigr)
\rightarrow R^{s}\pi_{*}\mathscr{O}_{E}\bigl((-j+1)E\bigr)\rightarrow \cdots,
\]
where $E$ is the exceptional divisor,
and the fact that $R^{s}\pi_{*}\mathscr{O}_{\Bl_Y X}(-jE)=0$ for $j\gg 0$ and $s>0$, thanks to the fact that $-E$ is $\pi$-ample.

By the $q$-ampleness of $\mathscr{L}|_{Y}$, 
we have
\[
E^{p,s}_2=H^{p}\bigl(X,R^{s}\pi_{*}\mathscr{O}_{\Bl_Y X}\otimes \mathscr{L}^{\otimes m}\otimes\mathscr{O}_X(-l)\bigr)=0
\]
for $p>q$, $s>0$ and $m\gg0$.

These two vanishing results imply that
$E^{p-h,h-1}_{h}=E^{p-h,h-1}_{2}=0$ for $h \geq 2$, $p>q+r$ and $m\gg0$.

By the hypothesis,
\[
E^{p,0}_{\infty}=H^{p}\biggl(\Bl_Y X, \pi^{*}\bigl(\mathscr{L}^{\otimes m}\otimes\mathscr{O}_X(-l)\bigr)\biggr)=0
\]
for $p>q+r$ and $m\gg 0$.
Since the differentials on the $h$-th page $E_{h}^{p-h,h-1}\rightarrow E^{p,0}_h$ and $E_{h}^{p,0}\rightarrow E^{p+h,-h+1}_h$ are zero maps for $h\geq 2$, we arrive at the desired vanishing $E^{p,0}_{2}=H^{p}\bigl(X, \mathscr{L}^{\otimes m}\otimes\mathscr{O}_X(-l)\bigr)=0$ for $p>q+r$ and $m\gg 0$. 
\end{proof}

\begin{lemma}\label{lemma: trans ample}
Let $Y$ (resp. $Z$) be a closed subscheme of $X$ (resp. Y). Then we have the following commutative diagram.
\[
\xymatrix{
\Bl_{\mathscr{I}_{Y}\cdot\mathscr{I}_{Z}}X
\ar@/^1.5pc/[drr]^{\pi_Y}
\ar@/_1.5pc/[ddr]_{\pi_Z}
\ar[dr]^{\iota} & & \\
& \Bl_{\mathscr{I}_{Y}}X \times_X \Bl_{\mathscr{I}_{Z}}X \ar[r]^-p \ar[d]_q
& \Bl_{\mathscr{I}_{Z}}X \ar[d]^{\pi'_Z} \\
& \Bl_{\mathscr{I}_{Y}}X \ar[r]^{\pi'_Y}
& X.
}
\]

Here $\pi'_Z$ (resp. $\pi'_Y$) is the blowup of $X$ along $\mathscr{I}_Z$ (resp. $\mathscr{I}_Y$), with exceptional divisor $E'_Z$ (resp. $E'_Y$);
$\pi_Z$ (resp. $\pi_Y$) is the blowup along the ideal sheaf 
$\mathscr{I}_Z\cdot \mathscr{O}_{\Bl_{\mathscr{I}_Y}X}$ (resp. 
$\mathscr{I}_Y\cdot \mathscr{O}_{\Bl_{\mathscr{I}_Z}X}$), with exceptional divisor $E_Z$ 
(resp. $E_Y$).
The square in the above diagram is the obvious fiber diagram, with $\iota$ induced by the maps $\pi_Z$ and $\pi_Y$.

Moreover,
\begin{enumerate}
\item\label{item:productc} The compositions $\pi_Y\circ\pi'_Z=\pi_Z\circ\pi'_Y$ coincide with the blowup of $X$ along $\mathscr{I}_Y\cdot\mathscr{I}_Z$.
    \item\label{item: producta} $\pi_Y^{*}E'_Z=E_Z$ and $\pi_Z^{*}E'_Y=E_Y$.
    \item\label{item: productb} $\iota$ is a closed immersion.
    \item\label{item: productd} 
    Let $\widetilde{Y}$ be the closed subscheme  defined by the ideal sheaf $\Bl_{\mathscr{I}_Z}X$.
    The morphism $\pi_Y$ may also be seen as the blowup along $\mathscr{I}_{\widetilde{Y}}$, with exceptional divisor $E_Y-E_Z$ instead.
    We denote $E_{\widetilde{Y}}:=E_Y-E_Z$.
\end{enumerate}
\end{lemma}
\begin{proof}
The proof is elementary and is omitted.
\end{proof}
\begin{theorem}[Transitivity of ample subschemes]\label{theorem: trans ample}
Let $Y\subset X$ be an ample subscheme of codimension $r_1$, $Z\subset Y$ be an ample subscheme of codimension $r_2$. Then $Z\subset X$ is also an ample subscheme of codimension $r_1+r_2$.
\end{theorem}
\begin{proof}

Note that $\pi_Y$ has fiber dimension at most $r_1 -1$.
This follows from (\ref{item: productb}) of Lemma \ref{lemma: trans ample}
and the fact that $\pi'_Y$ has fiber dimension at most $r_1-1$ (Proposition \ref{prop: equi}).

By Lemma \ref{lemma: pass to blowup} and the fact that $\pi_Y^{*}E'_Z=E_Z$, it suffices to prove that given any $l\in \mathbb{Z}_{\geq0}$,
\begin{equation}\label{eq:vanishing99}
H^{i}\biggl(\Bl_{\mathscr{I}_{Y}\cdot\mathscr{I}_{Z}}X,\mathscr{O}_{\Bl_{\mathscr{I}_{Y}\cdot\mathscr{I}_{Z}}X}(mE_Z)\otimes\pi_Y^*\mathscr{O}_{\Bl_{\mathscr{I}_{Z}}X}(-lH)\biggr)=0
\end{equation}
for $i>r_1 +r_2 -1$ and $m\gg 0$.
Here $H$ is an ample divisor on $\Bl_{\mathscr{I}_{Z}}X$.
We fix an $l\in \mathbb{Z}_{\geq0}$ from now on.

\begin{claim}
$(k E_Z - E_{\widetilde{Y}})|_{E_{\widetilde{Y}}}$ is $(r_2-1)$-ample for $k\gg 0$.
\end{claim}
\begin{proof}[Proof of claim]
Since $Z$ is an ample subscheme of $Y$, $E'_Z|_{\widetilde{Y}}$ is $(r_2 -1)$-ample.
Recall that $E_{\widetilde{Y}}$ is the exceptional divisor of $\pi_Y$, considered as the blowup along $\widetilde{Y}$.
It follows that $-E_{\widetilde{Y}}$ is $\pi_Y$-ample.
Thus, $(\pi_Y^*kE'_Z - E_{\widetilde{Y}})|_{E_{\widetilde{Y}}}=(kE_Z -E_{\widetilde{Y}})|_{E_{\widetilde{Y}}}$ is $(r_2-1)$-ample, for $k\gg 0$, by Proposition \ref{prop: pullback}.
\end{proof}

\begin{claim}
$kE_Z + (k+1)E_{\widetilde{Y}}=(k+1)E_Y-E_Z$ is $(r_1-1)$-ample for $k\gg 0$.
\end{claim}
\begin{proof}[Proof of claim]
Indeed, $E_Y=\pi_{Z}^{*}E'_Y$ and $E'_Y$ is $(r_1 -1)$-ample by the ampleness of $Y\subset X$.
Note that $-E_Z$ is $\pi_{Z}$-ample.
The claim follows again from Proposition \ref{prop: pullback}.
\end{proof}

Fix a large enough $k\in\mathbb{Z}$ such that both claims hold.
Given positive integers $m_1$ and $m_2$, we may write 
\[
m_1 E_{\widetilde{Y}}+m_2 E_Z =\lambda_1(kE_Z - E_{\widetilde{Y}}) +\lambda_2(kE_Z +(k+1)E_{\widetilde{Y}}) + j_1 E_{\widetilde{Y}} +j_2 E_Z,
\]
where $\lambda_2=\lfloor\frac{m_1+\lfloor  \frac{m_2}{k}\rfloor}{k+2}\rfloor$;
$\lambda_1= \lfloor\frac{m_2}{k}\rfloor-\lambda_2$;
$j_1=(m_1+\lfloor\frac{m_2}{k}\rfloor)\bmod (k+2)$
and
$j_2= m_2\bmod k.$
Note that $0\leq j_1<k+2$ and $0\leq j_2<k$.
The plan is to fix a big $m_2$, and then let $m_1$ increase. 
As $m_1$ grows, $\lambda_1$ decreases while $\lambda_2$ increases.
We then use the positivity of $(kE_Z- E_{\widetilde{Y}})|_{E_{\widetilde{Y}}}$ and $kE_Z + (k+1)E_{\widetilde{Y}}$ to prove the required vanishing statement (\ref{eq:vanishing99}).

Since $kE_Z +(k+1)E_{\widetilde{Y}}$ is $(r_1-1)$-ample, we may find $\Lambda_2$ such that 
\begin{equation}\label{eq: a}
H^{i}\biggl(\Bl_{\mathscr{I}_{Y}\cdot\mathscr{I}_{Z}}X,\mathscr{O}\bigl(\lambda_2(kE_Z +(k+1)E_{\widetilde{Y}}) + j_1 E_{\widetilde{Y}} +j_2 E_Z\bigr)\otimes\pi_Y^*\mathscr{O}_{\Bl_{\mathscr{I}_{Z}}X}(-lH)\biggr)=0
\end{equation}
for $i>r_1 -1$, $\lambda_2\geq \Lambda_2$, $0\leq j_1<k+2$ and $0\leq j_2 <k$.

Applying Theorem \ref{theorem: fujita} to the subscheme $E_{\widetilde{Y}}\subset\Bl_{\mathscr{I}_{Y}\cdot\mathscr{I}_{Z}}X$,
this gives a positive integer $\Lambda'_2$ such that 
\[
H^{i}\biggl(E_{\widetilde{Y}}, \mathscr{O}_{E_Y}\bigl(\lambda_1(kE_Z - E_{\widetilde{Y}}) +\lambda_2(kE_Z +(k+1)E_{\widetilde{Y}}) + j_1 E_{\widetilde{Y}} +j_2 E_Z\bigr)\otimes \pi^{*}_Y\mathscr{O}_{\Bl_{\mathscr{I}_Z}}(-lH)\biggr)=0
\]
for $i>(r_2 -1) + (r_1 -1)$, $\lambda_1\geq 0$, $\lambda_2 \geq \Lambda'_2$, $0\leq j_1<k+2$ and $0\leq j_2<k$.
Chasing through the obvious long exact sequence, this gives
\begin{multline}\label{eq: b}
H^{i}\biggl(\Bl_{\mathscr{I}_{Y}\cdot\mathscr{I}_{Z}}X,\mathscr{O}(m_2 E_Z + m_1 E_{\widetilde{Y}})\otimes\pi_Y^*\mathscr{O}_{\Bl_{\mathscr{I}_{Z}}X}(-lH)\biggr)\\
\cong
H^{i}\biggl(\Bl_{\mathscr{I}_{Y}\cdot\mathscr{I}_{Z}}X,\mathscr{O}\bigl(m_2 E_Z + (m_1 +1)E_{\widetilde{Y}}\bigr)\otimes\pi_Y^*\mathscr{O}_{\Bl_{\mathscr{I}_{Z}}X}(-lH)\biggr)
\end{multline}
for $i>r_1 + r_2 -1$, $0<m_1+1< (k+1)\lfloor\frac{m_2}{k}\rfloor$ and $\lfloor\frac{m_1 +1+\lfloor  \frac{m_2}{k}\rfloor}{k+2}\rfloor\geq \Lambda'_2$.

Choose some big $M_2$ such that $\lfloor\frac{\lfloor  \frac{M_2}{k}\rfloor}{k+2}\rfloor\geq \max\{\Lambda_2,\Lambda'_2\}$.
Applying (\ref{eq: b}) repeatedly, increasing $m_1$ by $1$ at a time, starting from $0$,
we have for $m_2>M_2$, 
\begin{multline*}
H^{i}\biggl(\Bl_{\mathscr{I}_{Y}\cdot\mathscr{I}_{Z}}X,\mathscr{O}(m_2 E_Z)\otimes\pi_Y^*\mathscr{O}_{\Bl_{\mathscr{I}_{Z}}X}(-lH)\biggr)\\
\cong
H^{i}\biggl(\Bl_{\mathscr{I}_{Y}\cdot\mathscr{I}_{Z}}X,\mathscr{O}\bigl(m_2 E_Z + (k+1)\lfloor \frac{m_2}{k}\rfloor E_{\widetilde{Y}}\bigr)\otimes\pi_Y^*\mathscr{O}_{\Bl_{\mathscr{I}_{Z}}X}(-lH)\biggr)
\end{multline*}
for $i>r_1 + r_2 -1$.
The above cohomology group can be rewritten as
\[
H^{i}\biggl(\Bl_{\mathscr{I}_{Y}\cdot\mathscr{I}_{Z}}X,\mathscr{O}\bigl(\lfloor\frac{m_2}{k}\rfloor(kE_Z +(k+1)E_{\widetilde{Y}}) +(m_2 -k\lfloor\frac{m_2}{k}\rfloor) E_Z\bigr)\otimes\pi_Y^*\mathscr{O}_{\Bl_{\mathscr{I}_{Z}}X}(-lH)\biggr),
\]
which is $0$ by (\ref{eq: a}).
\end{proof}

We then prove the analogue of Theorem \ref{theorem: trans ample} for nef subschemes.
The idea of the proof is essentially the same,
although we have to use the full statement of Theorem \ref{theorem: fujita} by allowing a nef term, as well as taking extra care with the variables.
\begin{theorem}[Transitivity of nef subschemes]\label{theorem: trans nef}
Let $Y\subset X$ be a nef subscheme of codimension $r_1$, $Z\subset Y$ be a nef subscheme of codimension $r_2$. Then $Z\subset X$ is also a nef subscheme of codimension $r_1+r_2$.
\end{theorem}
\begin{proof}
We shall use the same notation as in Lemma \ref{lemma: trans ample}.
Since $-E'_Z$ and $-E'_Y$ are $\pi'_Z$-ample and $\pi'_Y$-ample respectively, we may choose an ample divisor $A'$ on $X$ such that $\pi_Z^{\prime*} A'-E'_Z$ and $\pi^{\prime *}_Y A'-E'_Y$ are ample.
Let $A=\pi_Y^*\pi_Z^{\prime*}A'$ be the pullback of $A$ to $\Bl_{\mathscr{I}_Y\cdot\mathscr{I}_Z}X$.
Note that $A$ is nef.

Note that we can write $kE_Z+ A$ as $\pi_Y^{*}\bigl((k+1)E'_Z+(\pi_Z^{\prime*} A'-E'_Z)\bigr)$.
By Lemma \ref{lemma: pass to blowup},
it suffices to prove that given any integer $l\geq0$ and $k\gg0$,
\begin{equation}\label{eq:4.12}
H^{i}\biggl(\Bl_{\mathscr{I}_{Y}\cdot\mathscr{I}_{Z}}X,\mathscr{O}\bigl(m_2(kE_Z+ A)\bigr)\otimes\pi_Y^*\mathscr{O}_{\Bl_{\mathscr{I}_{Z}}X}(-l)\biggr)=0
\end{equation}
for $i>r_1 +r_2 -1$ and $m_2\gg 0$. We fix $l$ and $k$ from this point on.

Fix a large $k_1$ such that $\frac{1}{3k}\pi^{*}_Y(\pi_Z^{\prime *}A'-E'_Z)-\frac{1}{k_1}E_{\widetilde{Y}}$ and
$\frac{1}{3k}\pi_Z^{*}(\pi_Y^{\prime *}A' -E'_Y)-\frac{1}{k_1}E_Z$
are both ample.
It follows that $F'_1:=E_Z + \frac{1}{3k}\pi^{*}_Y(\pi_Z^{\prime *}A'-E'_Z)-\frac{1}{k_1}E_{\widetilde{Y}}$ is $(r_2-1)$-ample when restricted to $E_{\widetilde{Y}}$ and 
$F'_2:=E_Z +E_{\widetilde{Y}} + \frac{1}{3k}\pi_Z^{*}(\pi_Y^{\prime *}A' -E'_Y)-\frac{1}{k_1}E_Z$ is $(r_1 -1)$-ample.
Let $\alpha = 3kk_1 -k_1$ and $\beta = 3kk_1-k_1-3k$.
Let $F_1$ and $F_2$ be integral divisors
$3kk_1\beta F'_1$ and $F_2=3kk_1\alpha F'_2$ respectively.
We may express them as
\[
F_1=\beta(\alpha E_Z-3kE_{\widetilde{Y}}+k_1A)
\]
and
\[
F_2=\alpha(\beta E_Z+\alpha E_{\widetilde{Y}} +k_1A).
\]

Given nonnegative integers $m_1$ and $m_2$, we can express $m_1 E_{\widetilde{Y}}+m_2 (kE_Z+A)$ in terms of $F_1$, $F_2$ and $A$ with some small remaining terms $j_1 E_{\widetilde{Y}} +j_2 E_Z$:
\[
m_1 E_{\widetilde{Y}}+m_2 (kE_Z+A) =\lambda_1 F_1 +\lambda_2 F_2 + \lambda_3 A+ j_1 E_{\widetilde{Y}} +j_2 E_Z,
\]
where the $\lambda$'s and $j$'s are integer-valued functions depending on $m_1$ and $m_2$, with 
\[
\lambda_2=\lfloor\frac{m_1 + 3\beta k\lfloor\frac{m_2 k}{\alpha \beta}\rfloor}{\alpha^2 +3\beta k}\rfloor;
\]
\[
\lambda_1= \lfloor\frac{m_2 k}{\alpha\beta}\rfloor-\lambda_2;
\]
\[
\lambda_3=m_2 -\lambda_1\beta k_1-\lambda_2\alpha k_1;
\]
$j_1=m_1 + 3\beta k \lfloor\frac{m_2 k}{\alpha \beta}\rfloor \bmod (\alpha^2 + 3\beta k)$
and $j_2=m_2 k \bmod \alpha\beta$.
It follows that 
\begin{equation}\label{eq: lambda}
\text{if }0\leq m_1\leq \alpha^2\lfloor\frac{m_2 k}{\alpha \beta}\rfloor\text{, then } \lambda_1\geq 0 \text{ and }\lambda_3\geq 0.
\end{equation}

Since $F_2$ is $(r_1 -1)$-ample and $A$ is nef, 
we may apply Theorem \ref{theorem: fujita}
which gives us a positive integer $\Lambda_2$ such that
\begin{equation}\label{eq: trans nef a}
H^{i}\biggl(\Bl_{\mathscr{I}_{Y}\cdot\mathscr{I}_{Z}}X, \mathscr{O}(\lambda_2 F_2 + \lambda_3 A +j_2 E_Z)\otimes \pi_Y^*\mathscr{O}_{\Bl_{\mathscr{I}_{Z}}X}(-l)\biggr)=0
\end{equation}
for $i>r_1 -1$, $\lambda_2>\Lambda_2$,
$\lambda_3\geq0$ and 
$0\leq j_2<\alpha\beta$.

Since $F_1|_{E_{\widetilde{Y}}}$ is $(r_2 -1)$-ample,
$F_2$ is $(r_1-1)$-ample and
$A$ is nef,
we may apply Theorem \ref{theorem: fujita} again, which gives us a positive integer $\Lambda'_2$ such that
\begin{equation}\label{eq:4.12a}
H^{i}\biggl(E_{\widetilde{Y}}, \mathscr{O}_{E_{\widetilde{Y}}}(\lambda_1 F_1 +\lambda_2 F_2 + \lambda_3 A+ j_1 E_{\widetilde{Y}} +j_2 E_Z)\otimes \pi_Y^*\mathscr{O}_{\Bl_{\mathscr{I}_{Z}}X}(-l)\biggr)=0
\end{equation}
for $i>(r_2 -1) + (r_1 -1)$, $\lambda_2>\Lambda'_2$, $\lambda_1\geq 0$, $\lambda_3\geq 0$, $0\leq j_1<\alpha^2 + 3\beta k$ and $0\leq j_2<\alpha\beta$.

We tensor $\mathscr{O}\bigl(m_2(kE_Z+ A)\bigr)\otimes\pi_Y^*\mathscr{O}_{\Bl_{\mathscr{I}_{Z}}X}(-l)$ with the short exact sequence
\[
0\rightarrow
\mathscr{O}_{\Bl_{\mathscr{I}_{Y}\cdot\mathscr{I}_{Z}}X}(m_1E_{\widetilde{Y}})\rightarrow
\mathscr{O}_{\Bl_{\mathscr{I}_{Y}\cdot\mathscr{I}_{Z}}X}((m_1+1)E_{\widetilde{Y}})
\rightarrow
\mathscr{O}_{E_{\widetilde{Y}}}((m_1+1)E_{\widetilde{Y}})\rightarrow
0
\]
with $\lfloor\frac{3\beta k\lfloor\frac{m_2 k}{\alpha \beta}\rfloor}{\alpha^2 +3\beta k}\rfloor>\Lambda'_2$ and $m_1$ ranges over $0$ to $\alpha^2\lfloor\frac{m_2 k}{\alpha \beta}\rfloor-1$,
then consider the associated long exact sequence.
Thanks to (\ref{eq:4.12a}) and (\ref{eq: lambda}),
it follows that
\begin{multline*}
H^{i}\biggl(\Bl_{\mathscr{I}_{Y}\cdot\mathscr{I}_{Z}}X,\mathscr{O}\bigl(m_2(kE_Z+A)\bigr)\otimes\pi_Y^*\mathscr{O}_{\Bl_{\mathscr{I}_{Z}}X}(-l)\biggr)\\
\cong
H^{i}\biggl(\Bl_{\mathscr{I}_{Y}\cdot\mathscr{I}_{Z}}X,\mathscr{O}\bigl(\alpha^2\lfloor\frac{m_2 k}{\alpha \beta}\rfloor E_{\widetilde{Y}} + m_2(kE_Z+A)\bigr)\otimes\pi_Y^*\mathscr{O}_{\Bl_{\mathscr{I}_{Z}}X}(-l)\biggr)
\end{multline*}
for $i>r_1 + r_2 -1$ and $\lfloor\frac{3\beta k\lfloor\frac{m_2 k}{\alpha \beta}\rfloor}{\alpha^2 +3\beta k}\rfloor>\Lambda'_2$.
Note that the left hand side of the above isomorphism is the same as the cohomology group that appears in (\ref{eq:4.12}).

The above cohomology group can be rewritten as
\[
H^{i}\biggl(\Bl_{\mathscr{I}_{Y}\cdot\mathscr{I}_{Z}}X, \mathscr{O}(\lfloor\frac{m_2 k }{\alpha \beta}\rfloor F_2 + \lambda_3 A +j_2 E_Z)\otimes \pi_Y^*\mathscr{O}_{\Bl_{\mathscr{I}_{Z}}X}(-l)\biggr)
\]
for some integers $\lambda_3\geq0$ and $0\leq j_2<\alpha\beta$.
By (\ref{eq: trans nef a}), the above cohomology group vanishes for $m_2\gg0$.
\end{proof}

The following corollary says that the intersection of $2$ ample (resp. nef) subschemes is ample (resp. nef),
assuming the intersection has the desired codimension.
It generalizes
\cite[Proposition 6.3]{Ottem},
in the sense that we do not assume that $X$ is smooth and the subschemes are l.c.i. in $X$.
\begin{corollary}[Intersection of ample or nef subschemes]\label{cor: intersection of ample}
If $Y$ and $Z$ are both ample (resp. nef) subschemes of $X$, of codimension $r$ and $s$ respectively and $Y\cap Z$ has codimension $r+s$ in $X$, then $Y\cap Z$ is an ample (resp. nef) subscheme of $X$.
\end{corollary}
\begin{proof}
By Proposition \ref{prop: ample intersect}, $Y\cap Z$ is an ample (resp. nef) subscheme of $Z$. 
We now conclude using the transitivity property of ample (resp. nef) subschemes (Theorem \ref{theorem: trans ample} and Theorem \ref{theorem: trans nef} respectively).
\end{proof}

\section{Positivity of a line bundle upon restriction to an ample subscheme}
If a line bundle is ample after restricting to an ample subscheme,
it is reasonable to expect the line bundle to exhibit some positivity features in the ambient space.
The following theorem demonstrates an interplay between ample subschemes and $q$-ample divisors.

\begin{theorem}\label{theorem: restriction}
Let $X$ be a projective scheme of dimension $n$ and $Y$ be an ample subscheme of $X$ of codimension $r$.
Suppose $\mathscr{L}$ is a line bundle on $X$, and that its restriction $\mathscr{L}|_{Y}$ to $Y$ is $q$-ample. Then $\mathscr{L}$ is $(q+r)$-ample.
\end{theorem}
\begin{proof}
We fix an ample line bundle $\mathscr{O}_X(1)$ on $X$.
Let $\pi: \widetilde{X}\rightarrow X$ be the blowup of $X$ along $Y$, with $E$ the exceptional divisor.

\begin{step}\label{step: blowup} Pass to the blowup.
\end{step}
 
By Lemma \ref{lemma: pass to blowup}, it suffices to prove that for any integer $l\geq0$,
\begin{equation}\label{blowup}
H^{i}\biggl(\widetilde{X}, \pi^{*}\bigl(\mathscr{L}^{\otimes m}\otimes\mathscr{O}_X(-l)\bigr)\biggr)=0
\end{equation}
for $i>q+r$ and $m\gg 0$.
From now on, we fix $l\geq 0$.

\begin{step}\label{step: exceptional}
Pass to the exceptional divisor.
\end{step}

We claim it is enough to show that there is a positive integer $m_0$ such that
\begin{equation}\label{exceptional}
H^{i}\biggl(E,\pi^{*}\bigl(\mathscr{L}^{\otimes m}\otimes\mathscr{O}_X(-l)\bigr)\otimes\mathscr{O}_{E}(kE)\biggr)=0
\end{equation}
for $i>q+r-1$, $m>m_0$ and $k\geq 1$. 

Indeed, let us consider the following short exact sequence on $\widetilde{X}$:
\begin{multline*}
0\rightarrow \pi^{*}\bigl(\mathscr{L}^{\otimes m}\otimes\mathscr{O}_X(-l)\bigr)\otimes\mathscr{O}_{\widetilde{X}}\bigl((k-1)E\bigr) 
\rightarrow
\pi^{*}\bigl(\mathscr{L}^{\otimes m}\otimes\mathscr{O}_X(-l)\bigr)\otimes\mathscr{O}_{\widetilde{X}}(kE) 
\\
\rightarrow
\pi^{*}\bigl(\mathscr{L}^{\otimes m}\otimes\mathscr{O}_X(-l)\bigr)\otimes\mathscr{O}_{E}(kE)
\rightarrow
0.
\end{multline*}

By looking at the long exact sequence of cohomology groups induced from the above short exact sequence and using Hypothesis (\ref{exceptional}), we observe that
\begin{equation}\label{iso}
H^{i}\biggl(\widetilde{X},\pi^{*}\bigl(\mathscr{L}^{\otimes m}\otimes\mathscr{O}_X(-l)\bigr)\otimes\mathscr{O}_{\widetilde{X}}\bigl((k-1)E)\bigr)\biggr)
\cong H^{i}\biggl(\widetilde{X},\pi^{*}\bigl(\mathscr{L}^{\otimes m}\otimes\mathscr{O}_X(-l)\bigr)\otimes\mathscr{O}_{\widetilde{X}}(kE)\biggr)
\end{equation}
for $i>q+r$, $m>m_0$ and $k\geq 1$.

Since $E$ is $(r-1)$-ample, 
for any fixed $m$,
\[
H^{i}\biggl(\widetilde{X},\pi^{*}\bigl(\mathscr{L}^{\otimes m}\otimes\mathscr{O}_X(-l)\bigr)\otimes\mathscr{O}_{\widetilde{X}}(kE)\biggr)=0
\]
for $k\gg0$ and $i>r-1$. Together with the isomorphism in (\ref{iso}), we have the desired vanishing result (\ref{blowup}).

\begin{step} 
Rewrite the line bundles of interest in ($\ref{exceptional}$) in terms of $q$-ample and $(r-1)$-ample line bundles.
\end{step}

Note that $-E$ is $\pi$-ample.
By Proposition \ref{prop: pullback},
there is an $N>0$ 
such that $\pi^{*}\mathscr{L}^{\otimes N}\otimes \mathscr{O}_{E}(-E)$ is $q$-ample. 
We can replace $\mathscr{L}$ by $\mathscr{L}^{\otimes N}$ and assume that $\pi^{*}\mathscr{L}\otimes \mathscr{O}_E(-E)$ is $q$-ample. 
We now rewrite the line bundle on $E$ in (\ref{exceptional}):
\[
\pi^{*}\bigl(\mathscr{L}^{\otimes m}\otimes\mathscr{O}_X(-l)\bigr)\otimes\mathscr{O}_E(kE)
\cong \pi^{*}\mathscr{O}_X(-l)
\otimes\mathscr{O}_E\bigl((k+m)E\bigr)
\otimes\bigl(\pi^{*}\mathscr{L}\otimes \mathscr{O}_E(-E)\bigr)^{\otimes m}
\]
with the second term $\mathscr{O}_E\bigl((k+m)E\bigr)$ on the right hand side being an $(r-1)$-ample line bundle,
and the third term $\bigl(\pi^{*}\mathscr{L}\otimes \mathscr{O}_E(-E)\bigr)^{\otimes m}$ being a $q$-ample line bundle.

We now apply Theorem \ref{theorem: fujita} with $\mathscr{L}_1:=\pi^{*}\mathscr{L}\otimes \mathscr{O}_E(-E)$, $\mathscr{L}_2=\mathscr{O}_E(E)$ and $M_2=1$ to conclude.
\end{proof}

One may ask whether we have a converse to Theorem \ref{theorem: restriction}, i.e.,
given an $r$-ample line bundle $\mathscr{L}$ on a projective scheme $X$, is there a codimension $r$ ample subscheme $Y$, such that $\mathscr{L}|_{Y}$ is ample?
Demailly, Peternell and Schneider gave a counter-example to this in \cite[Example 5.6]{DPS}:
\begin{example}\label{example: DPS}
Let $S$ be a general quartic surface in $\mathbb{P}^3$.
Let $X=\mathbb{P}(\mathscr{T}_{S})$, where $\mathscr{T}_{S}$ is the tangent bundle of $S$.
They showed that $-K_X$ is $1$-ample, and yet for any ample divisor $Y$ in $X$,
$(-K_X)^2\cdot Y<0$, thus $-K_X$ cannot be ample when it is restricted to any ample divisor.
\end{example}
For the reader's convenience, we shall include the proof of $-K_X$ being $1$-ample in Example \ref{eg:kodaira}. 
In fact, it might be worthwhile to extract from the argument of \cite[Example 5.6]{DPS} the following general property.

\begin{proposition}\label{prop:ses}
Let 
\begin{equation}\label{ses vector bundle}
0\rightarrow \mathscr{E}' \rightarrow \mathscr{E} \rightarrow \mathscr{L} \rightarrow 0
\end{equation}
be a short exact sequence of vector bundles on a projective scheme $X$. 
We assume $\mathscr{E}$ to be a $q$-ample vector bundle of rank $r$, $\mathscr{E}'$ is of rank $(r-1)$ and $\mathscr{L}$ is of rank $1$.
Then $\mathscr{E}'$ is $(q+1)$-ample.

By $q$-ampleness of a vector bundle $\mathscr{E}$, we meant $\mathscr{O}_{\mathbb{P}(\mathscr{E})}(1)$ is $q$-ample, or equivalently
for any coherent sheaf $\mathscr{F}$ on $X$, there is a positive integer $m_0$ such that for any $m\geq m_0$ and $q>i$,
\[
H^{i}(X,\Sym^{m}\mathscr{E}\otimes\mathscr{F})=0.
\]
\end{proposition}
\begin{proof}
We first dualize (\ref{ses vector bundle}), then take symmetric product \cite[Lemma 17.19.4]{Stacks} to get a right exact sequence:
\[
\Sym^{k-1} \mathscr{E}^{\vee}\otimes \mathscr{L}^{\vee}
\rightarrow
\Sym^k \mathscr{E}^{\vee}
\rightarrow
\Sym^k \mathscr{E}'^{\vee}\rightarrow 0.
\]
It is easy to see that $\rk(\Sym^{k-1} \mathscr{E}^{\vee}\otimes \mathscr{L}^{\vee})
=
\rk(\Sym^k \mathscr{E}^{\vee})-
\rk(\Sym^k \mathscr{E}'^{\vee})$.
Therefore the above exact sequence is left exact as well.
Dualizing, this gives us the following short exact sequence
\[
0\rightarrow \Sym^k \mathscr{E}'
\rightarrow \Sym^k \mathscr{E}
\rightarrow \Sym^{k-1} \mathscr{E}\otimes \mathscr{L}
\rightarrow 0.
\]
Fix an ample line bundle $\mathscr{O}_X(1)$ on $X$, and tensor the above short exact sequence with $\mathscr{O}_X(-l)$, for $l\geq0$.
Then consider the associated long exact sequence.
Note that 
$H^{i}\bigl(X,\Sym^k\mathscr{E}\otimes\mathscr{O}_X(-l)\bigr)= H^{i}\bigl(X,\Sym^{k-1}\mathscr{E}\otimes \mathscr{L}\otimes\mathscr{O}_X(-l)\bigr)=0$,
for $i>q$ and $k\gg 0$.
Hence $H^{i}\bigl(X,\Sym^k\mathscr{E}'\otimes \mathscr{O}_X(-l)\bigr)=0$, for $i>q+1$ and $k\gg0$.
\end{proof}

\begin{example}\label{eg:kodaira}
Let $S\subset \mathbb{P}^{n+1}$ be a smooth hypersurface of degree $(n+2)$, with $\dim S=n\geq 2$
($n=2$ in Example \ref{example: DPS}).
Let $X=\mathbb{P}(\mathscr{T}_S)$, where $\mathscr{T}_S$ is the tangent bundle of $S$.
Then
\[
\mathscr{O}_X(-K_X)=\mathscr{O}_{\mathbb{P}(\mathscr{T}_S)}(n).
\]
Applying Proposition \ref{prop:ses} to the normal bundle exact sequence:
\[
0\rightarrow \mathscr{T}_S \rightarrow
\mathscr{T}_{\mathbb{P}^{n+1}}|_{S} \rightarrow
\mathscr{O}_S(S) \rightarrow 0,
\]
it follows that the tangent bundle of $S$ is $1$-ample,
hence $\mathscr{O}_X(-K_X)$ is $1$-ample as well.
Note that $\mathscr{T}_S$ is not ample,
as its top exterior power is trivial.
Therefore, $\mathscr{O}_X(-K_X)$ is not ample.
%
%

Ottem gave a counterexample to Kodaira-type vanishing theorem for $q$-ample divisors \cite[Chapter 9]{Ottem}:
More specifically,
he found a $1$-ample line bundle $\mathscr{L}$ on a smooth threefold $X$
such that $H^{2}(X,\omega_X\otimes \mathscr{L})\neq 0$.
In the above example, we note that 
\[
H^{n}(X,K_X -K_X)\cong
H^{n}(S,\mathscr{O}_S)
\cong 
H^{0}(S,\mathscr{O}_S(K_S))^{\vee}
\cong
H^{0}(S,\mathscr{O}_S)^{\vee}\neq 0,
\]
but $-K_X$ is $1$-ample. Thus, we obtain another example where Kodaira-type vanishing theorem fails for $q$-ample divisors.

However,
Greb and K\"uronya
\cite[Theorem 3.8]{GK}
showed that Kodaira-type vanishing theorem holds for \textit{big} $q$-ample $\mathbb{Q}$-divisors.
On toric varieties, Broomhead, Ottem and Prendergast-Smith proved that Kodaira-type vanishing theorem holds for $q$-ample divisors \cite[Theorem 6.1]{BOP}.
\end{example}

\begin{example}
One may ask whether we can relax the positivity assumption on $Y$ in Theorem \ref{theorem: restriction}. 
For example, if we only assume that the normal bundle of $Y$ is ample,
we shall see the conclusion of the theorem does not hold in general.
Let us start with a smooth ample subvariety $Y\subset X$ of a smooth projective variety.
We blow up a closed point $p$ in $X\setminus Y$.
Observe that the normal bundle of $Y\subset \Bl_{p} X$ is still ample.
Let $E\cong\mathbb{P}^{n-1}$ be the exceptional divisor, and let $A$ be an ample divisor on $\Bl_{p}(X)$.
Then $E+\epsilon A$ is not $(n-2)$-ample, for $0<\epsilon\ll1$, since it is anti-ample when restricted to the exceptional divisor.
But $(E+\epsilon A)|_Y=\epsilon A|_{Y}$ is ample.

On the other hand, as we shall see in the following section, a small yet interesting part of the theorem still holds if we assume $Y$ is a nef subvariety.
\end{example}

\section{Restriction of a pseudoeffective divisor to a nef subvariety}\label{s6}

There are not many results regarding the positivity of subvariety with nef normal bundle, in terms of intersection theory. 
Here are two of such results the author is aware of.

In Fulton-Lazarsfeld's work \cite{FulLaz} (see also \cite[Theorem 8.4.1]{Laz}),
they proved that if $Y$ is a closed, l.c.i. subvariety of a projective variety $X$ and the normal bundle of $Y$ is nef, then for any closed subscheme $Z\subset X$ with $\dim Y + \dim Z \geq \dim X$, $\deg_H(Y\cdot Z)\geq 0$.
(Here $H$ is an ample divisor on $X$.)
On the other hand, it is not hard to show that if $Y$ has globally generated normal bundle, then restriction of any effective cycle to $Y$ is either effective or $0$ \cite[Theorem 12.1.a)]{Fulton}.

We now show that the restriction of a pseudoeffective divisor to a nef subvariety is still pseudoeffective. 

\begin{theorem}\label{theorem: pseudoeffective}
Let $Y$ be a nef subvariety of codimension $r$ of a projective variety $X$. 
Then 
\[
\iota^{*}\overline{\Eff}^{1}(X)\subseteq\overline{\Eff}^{1}(Y)
\]
and
\[
\iota^{*}\Big1 (X)\subseteq\Big1 (Y).
\]
Here $\iota:Y\hookrightarrow X$ is the inclusion map,
$\iota^{*}:\N^1(X)_{\mathbf{R}}\rightarrow\N^1(Y)_{\mathbf{R}}$ is the induced map on the N\'eron-Severi group with $\mathbf{R}$-coefficients and $\overline{\Eff}^{1}(X)$ (resp. $\Big1(X)$) is the cone of pseudoeffective (resp. big) $\mathbf{R}$-Cartier $\mathbf{R}$-divisors.
\end{theorem}
\begin{remark}
Before proving the theorem, let us point out it is rather straightforward to obtain the conclusion from Theorem \ref{theorem: restriction} under the stronger assumption that $Y$ is an \underline{ample} subscheme of $X$ and both $X$ and $Y$ are varieties.
Let $D$ be a pseudoeffective divisor on $X$, i.e. $-D$ is not $(n-1)$-ample (Theorem \ref{theorem: n-1 ample}).
Suppose on the contrary $D|_{Y}$ is not pseudoeffective.
Then $-D|_{Y}$ is $(n-r-1)$-ample.
This gives a contradiction to Theorem \ref{theorem: restriction}.
\end{remark}
\begin{proof}[Proof of Theorem \ref{theorem: pseudoeffective}]
A divisor is big if and only if it can be written as the sum of a pseudoeffective divisor and an ample divisor.
Therefore, we can focus on the pseudoeffective case.
We shall follow the steps in the proof of Theorem \ref{theorem: restriction} closely.
Recall that a Cartier divisor $D$ is $(n-1)$-ample if and only if $-D$ is not pseudoeffective (Theorem \ref{theorem: n-1 ample}).
Therefore, it suffices to show that
given a line bundle $\mathscr{L}$ on $X$ such that $\mathscr{L}|_Y$ is $(n-r-1)$-ample,
$\mathscr{L}$ has to be $(n-1)$-ample.
Fix an ample line bundle $\mathscr{O}_{X}(1)$ on $X$.
Let $\widetilde{X}=\Bl_Y X$.
\setcounter{step}{0}
\begin{step}[Pass to the blowup]
By Lemma \ref{lemma: pass to blowup},
it suffices to show for any integer $l\geq 0$,
there is an integer $m_0$ such that $H^{n}\biggl(\widetilde{X},\pi^{*}\bigl(\mathscr{L}^{\otimes m}\otimes \mathscr{O}_{X}(-l)\bigr)\biggr)=0$ for $m\geq m_0$.
\end{step}

We now fix $l$. 
\begin{step}[Pass to the exceptional divisor]
It is enough to show that there is a positive integer $m_0$ such that
\[
H^{n-1}\biggl(E, \pi^{*}\bigl(\mathscr{L}^{\otimes m}\otimes \mathscr{O}_{X}(-l)\bigr)\otimes\mathscr{O}_E(kE)\biggr)=0
\]
for $m\geq m_0$ and $k\geq 1$.
\end{step}
We just have to repeat the argument in Step \ref{step: exceptional} in the proof of Theorem \ref{theorem: restriction}, i.e., consider the long exact sequence of cohomologies associated to
\begin{multline*}
0\rightarrow
\pi^{*}\bigl(\mathscr{L}^{\otimes m}\otimes\mathscr{O}_{X}(-l)\bigr)\otimes \mathscr{O}_{\widetilde{X}}\bigl((k-1)E\bigr)\rightarrow
\pi^{*}\bigl(\mathscr{L}^{\otimes m}\otimes\mathscr{O}_{X}(-l)\bigr)\otimes \mathscr{O}_{\widetilde{X}}(kE)\\
\rightarrow
\pi^{*}\bigl(\mathscr{L}^{\otimes m}\otimes \mathscr{O}_{X}(-l)\bigr)\otimes\mathscr{O}_E(kE)
\rightarrow 0.
\end{multline*}

Also note that for a fixed integer $m$,
\[
H^{n}\biggl(\widetilde{X},\pi^{*}\bigl(\mathscr{L}^{\otimes m}\otimes\mathscr{O}_{X}(-l)\bigr)\otimes \mathscr{O}_{\widetilde{X}}(kE)\biggr)=0
\]
for $k\gg 0$. 
Indeed, $E$ is $(n-1)$-ample ($-E$ is not pseudoeffective!).
\begin{step}[Rewrite the line bundle in question in terms of an $(n-r-1)$-ample line bundle and an $(r-1)$-almost ample line bundle]
\end{step}
Replacing $\mathscr{L}$ with $\mathscr{L}^{\otimes N}$ for $N$ large enough, we may assume $\pi^{*}\mathscr{L}\otimes\mathscr{O}_{E}(-E)$ is $(n-r-1)$-ample, by Proposition \ref{prop: pullback}.
Now we can write
\[
\pi^{*}\bigl(\mathscr{L}^{\otimes m}\otimes \mathscr{O}_{X}(-l)\bigr)\otimes \mathscr{O}_{E}(kE)
\cong
\pi^*\mathscr{O}_{X}(-l)\otimes
\bigl(\pi^{*}\mathscr{L}\otimes\mathscr{O}_{E}(-E)\bigr)^{\otimes m}
\otimes
\mathscr{O}_E\bigl((k+m)E\bigr).
\]

By Proposition \ref{lemma: uniform2}, there is a positive integer $m_0$ such that
\[
H^{n-1}\biggl(E,\pi^*\mathscr{O}_{X}(-l)\otimes
\bigl(\pi^{*}\mathscr{L}\otimes\mathscr{O}_{E}(-E)\bigr)^{\otimes m}\otimes
\mathscr{O}_E\bigl((k+m)E\bigr)\biggr)=0
\]
for $k\geq 1$ and $m\geq m_0$. This proves the theorem.

\end{proof}

\begin{example}
It is possible that the normal bundle of $Y$ in $X$ is antiample but the restriction of any pseudoeffective divisor on $X$ to $Y$ is still pseudoeffective.

Take a general hypersurface $X\subset \mathbb{P}^{n+1}$ of degree $d=2n-1$, where $n\geq 3$.
It is a classical result that $X$ contains some line $l$ (see for e.g. \cite[Corollary 2.2]{bor}).
The line $l$ has anti-ample normal bundle $\bigoplus\mathscr{O}(-1)$:
Indeed, the computation of \cite[Proof of Proposition 2.1]{bor} showed that the normal bundle of $l$ does not have global sections,
while the degree of the normal bundle is $-n+1$ by the adjunction formula.
On the other hand, it follows from Lefschetz hyperplane theorem that $\Pic(X)= \mathbb{Z}\cdot \mathscr{O}_X(1)$.
Hence, the intersection number of any pseudoeffective divisor on $X$ with $l$ must be non-negative.

%
\end{example}

Boucksom, Demailly, P{\u{a}}un and Peternell showed that the dual cone of the pseudoeffective cone is the cone of movable curves \cite{BDPP}. Hence we have the equivalent statement:
\begin{corollary}\label{cor: movable curves}
With the same assumptions as in Theorem \ref{theorem: pseudoeffective},
the map on the numerical equivalence
classes of $1$-cycles,
induced by the inclusion $\iota:Y\hookrightarrow X$,
$\iota_{*}:\N_1(Y)\rightarrow \N_1(X)$, restricts to $\iota_*: \overline{\Mov}_{1}(Y)\rightarrow \overline{\Mov}_{1}(X)$, where $\overline{\Mov}_{1}(Y)$ and $\overline{\Mov}_{1}(X)$ are the cones of movable curves in $Y$ and $X$ respectively.
\end{corollary}

We may also apply the adjunction formula to get
\begin{corollary}
If both $X$ and $Y$ are non-singular, $Y$ has nef normal bundle and $K_X$ is pseudoeffective, then $K_Y$ is also pseudoeffective.
If $K_X$ is big, then $K_Y$ is also big.
\end{corollary}
\begin{remark}
The first assertion in the above corollary follows also from \cite{BDPP} and the theory of deformations of rational curves.
More specifically, Boucksom-Demailly-P\u{a}un-Peternell showed that on a smooth projective variety $Z$,
$K_Z$ is pseudoeffective if and only if $Z$ is not uniruled.
If $Y$ is uniruled, take a smooth rational curve $C$ that covers $Y$.
By considering the short exact sequence of normal bundles on $C$, we see that the normal bundle of $C$ in $X$ is nef.
Thus, $X$ is uniruled.
\end{remark}

\section{Weakly movable cone}\label{s7}
We shall define and study the weakly movable cone.
In this section, we assume the ground field $k$ is algebraically closed and of characteristic zero.
On a smooth projective variety, we know that the movable cone of divisors is the smallest closed convex cone that contains all the pushforwards of nef divisors from $\pi:X_{\pi}\rightarrow X$,
where $\pi$ ranges over all projective, dominant and generically finite morphisms to $X$.
With this in mind, we define the weakly movable cone as the closure of the cone that is generated by the pushforward of cycles of nef subvariety via a proper dominant morphism.
We find that it contains the movable cone and satisfies some desirable intersection theoretic properties.

\begin{definition}[Cycles modulo numerical equivalence]
Let $X$ be a projective variety over $k$ and $\Z_d(X)$ be the abelian group of $d$-cycles on $X$ with $\mathbb{Z}$-coefficients.
A $d$-cycle $\alpha\in \Z_d(X)$ is defined to be \textit{numerically trivial} if
\[
\deg(P\cap \alpha)=0,
\]
for any weight $d$ homogeneous polynomial $P$ in Chern classes of finite set of vector bundles on $X$ \cite[Definition 19.1]{Fulton}.
The \textit{numerical group of $d$-cycles with integer coefficients}, $\N_d(X)_{\mathbb{Z}}$, is defined to be the quotient of $\Z_d(X)$ modulo numerically trivial $d$-cylces.
This is a free abelian group of finite rank \cite[Example 19.1.4]{Fulton}.
The \textit{numerical group of $d$-cycles}, $\N_d(X)$ is then defined to be
$\N_d(X)_{\mathbb{Z}}\otimes_{\mathbb{Z}}\mathbb{R}$.
\end{definition}

We then recall Fulger-Lehmann's definition \cite{FL} of a family of effective cycles.
\begin{definition}[Family of effective cycles]
Let $X$ be a projective variety over $k$.
A \textit{family of effective $d$-cycles} on $X$ with $\mathbb{Z}$-coefficient, 
$(g: U\rightarrow W)$,
consists of a closed reduced subscheme $\Supp U$ of $W\times_{k} X$, where $W$ is a variety over $k$;
a coefficient $a_i\in \mathbb{Z}_{>0}$ for each irreducible component $U_i$ of $\Supp U$;
and the projections $g_i: U_i\rightarrow W$ are proper and flat of relative dimension $d$.

Over a closed point $w\in W$, one can identify $g_i^{-1}(w)$ as a closed subscheme of $X$. Its fundamental cycle $[g_i^{-1}(w)]$ is a $d$-cycle of $X$.
We define the \textit{cycle theoretic fiber} over $w$ to be $\sum a_i [g_i^{-1}(w)]$.

We say that the family of effective $d$-cycles is \textit{irreducible} if $\Supp U$ is irreducible.
\end{definition}
\begin{definition}[Strictly movable cycles {\cite[Definition 3.1]{FL}}]
A family of effective $d$-cycles of $X$ $(g: U\rightarrow W)$ is \textit{strictly movable} if each of the irreducible components $U_i$ of $\Supp U$ dominates $X$ via the second projection.

An effective $d$-cycle of $X$ (with $\mathbb{Z}$-coefficient) is \textit{strictly movable} if it is the cycle theoretic fiber over a closed point of a strictly movable family of $d$-cycles on $X$.

The \textit{movable cone of $d$-cycles}
$\overline{\Mov}_d(X)\subset \N_{d}(X)$ is defined to be the closure of the convex cone generated by strictly movable $d$-cycles.
\end{definition}

\begin{proposition}\label{prop: irreducible}
The movable cone of $d$-cycles is the closure of the convex cone generated by irreducible, strictly movable $d$-cycles.
\end{proposition}
\begin{proof}
Suppose $\sum a_i Z_i$ is the cycle theoretic fiber over a closed point of a family of strictly movable $d$-cycles $(g: U\rightarrow W)$ with irreducible components $U_i$.
It suffices to show that $Z_i$ is algebraically equivalent to a sum of irreducible strictly movable $d$-cycles.
If the generic fiber of $p_i:U_i\rightarrow W$ is geometrically integral, then the fiber over a general closed point is also (geometrically) integral \cite[Th\'eor\`eme 9.7.7]{EGAIV},
and we are done.

Suppose the generic fiber of $p_i$ is not geometrically integral.
Let $\eta_W$ be the generic point of $W$, $\overline{k(\eta_W)}$ be the algebraic closure of $k(\eta_W)$
and $U_{ij}'\subset X_{\overline{k(\eta_W)}}$ be the irreducible components of $\Spec \overline{k(\eta_W)}\times_{\Spec k(\eta_W)}U_i$.
We may take a finite field extension $k(\eta_W)\subset K$,
such that the generators of the ideal sheaves of $U_{ij}'$ are defined over $K$.
Then all the irreducible components of $\Spec K\times_{\Spec k(\eta_W)}U_i$ are geometrically integral.
These components dominate the generic fiber of $p_i$.
Take a variety $V$ with function field $K$ such that the map $\Spec K\rightarrow \Spec k(\eta_W)$ extends to $V\rightarrow W$.
By generic flatness, we may replace $V$ by a smaller open set and assume that each irreducible component $U_{ij}$ of $V\times_W U_i$ is flat over $V$.
Note that all $U_{ij}$ dominate $U_i$, hence also $X$.
Thus, each $U_{ij}$ is a strictly movable family of $d$-cycles of $X$ over $V$ (with coefficient $1$), 
and the cycle theoretic fiber over a general closed point of $V$ is (geometrically) integral, by \cite[Th\'eor\`eme 9.7.7]{EGAIV} again.
Then $Z_i$ is algebraically equivalent to the sum of the cycle theoretic fibers of $U_{ij}$'s, with $\mathbb{Z}$-coefficients, over a general closed point of $V$.
\end{proof}
\begin{proposition}\label{prop: movable is nef}
An irreducible, strictly movable cycle can be realized as the pushforward of a multiple of the cycle class of a nef subvariety via a proper, surjective morphism, up to algebraic equivalence.
\end{proposition}
\begin{proof}
From the proof of Proposition \ref{prop: irreducible}, 
we may assume that the irreducible, strictly movable cycle given is the cycle theoretic fiber over a closed point of an irreducible, strictly movable family of $(g: U\rightarrow W)$, with the fiber of $g':\Supp U \rightarrow W$ over a general closed point of $W$ integral.
Using the argument in \cite[Remark 2.13]{FL},
we may assume that $W$ is projective.
We note that a closed point $w\in W$ is nef, hence $g'^{-1}(w)$ is also nef, by Proposition \ref{prop: pullback of nef}, and that $g'^{-1}(w)$ is integral if $w$ is general.
\end{proof}

\begin{definition}[Weakly movable cone]
Let $X$ be a projective variety over $k$.
We define the \textit{weakly movable cone}
$\overline{\WMov}_d(X)\subset\N_d(X)$
to be the closure of the convex cone generated by 
$\pi_{*}[Z]$,
where $\pi:Y\rightarrow X$ is proper, surjective morphism from a projective variety and $Z$ is a nef subvariety of dimension $d$ in $Y$.
\end{definition}

We shall compare the movable cone and the weakly movable cone.

\begin{proposition}\label{prop: mov vs wmov}
Let $X$ be a projective variety over $k$. We have
\[
\overline{\Mov}_d(X)\subseteq\overline{\WMov}_d(X).
\]
In particular,
$\overline{\WMov}_d(X)$ is a full dimensional cone in $\N_d(X)$.
\end{proposition}
\begin{proof}
The first statement follows from Proposition \ref{prop: movable is nef}, while the second statement follows from \cite[Proposition 3.8]{FL}, which says that $\overline{\Mov}_d(X)$ is full dimensional.
\end{proof}

The following proposition is an analogue of the first statement of \cite[Lemma 3.6]{FL}.
\begin{proposition}
Let $X'$ and $X$ be projective variety over $k$.
Suppose $h:X'\rightarrow X$ is a surjective morphism.
Then $h_{*}\overline{\WMov}_d(X')\subseteq\overline{\WMov}_d(X)$.
\end{proposition}
\begin{proof}
It follows from the definition of the weakly movable cone.
\end{proof}
\begin{remark}
It is unclear whether the reverse inclusion is true.
On the other hand,
Fulger and Lehmann showed that $h_{*}\overline{\Mov}_d(X')=\overline{\Mov}_d(X)$ under the same assumptions \cite[Corollary 3.12]{FL}.
\end{remark}
The following theorem is an analogue of \cite[Lemma 3.10]{FL}.
\begin{theorem}\label{theorem: weakly movable}
Let $X$ be a projective variety over $k$ and
let $\alpha \in \overline{\WMov}_d(X)$. Then
\begin{enumerate}
    \item \label{item: eff}
    If $\beta\in \overline{\Eff}^1(X)$, then $\beta\cdot \alpha\in \overline{\Eff}_{d-1}(X)$.
    \item \label{item: bignef} Let $H$ be a big Cartier divisor.
    If $H\cdot \alpha=0$, then $\alpha =0$. 
    \item \label{item: nef}If $\beta\in \Nef^{1}(X)$ then $\beta\cdot \alpha\in \overline{\WMov}_{d-1}(X)$.
    
\end{enumerate}
\end{theorem}
\begin{proof}
For (\ref{item: eff}), we may assume that $\alpha=\pi_{*}[Z]$, where $\pi:Y\rightarrow X$ is a proper, surjective map and $Z$ a nef subvariety of $Y$.
By the projection formula, we have $\beta\cdot \pi_{*}[Z]=\pi_{*}(\pi^{*}\beta\cdot [Z])$. 
We know that $\pi^{*}\beta$ is pseudoeffective.
By Theorem \ref{theorem: pseudoeffective}, $\pi^{*}\beta\cdot [Z]\in \overline{\Eff}_{d-1}(Y)$.
Since $\pi_{*}\overline{\Eff}_{d-1}(Y)\subseteq \overline{\Eff}_{d-1}(X)$,
we have $\beta\cdot \pi_{*}[Z]\in\overline{\Eff}_{d-1}(X)$.

For (\ref{item: bignef}), we follow Fulger-Lehmann's argument \cite[Proof of Lemma 3.10]{FL}.
We write $H=A+E$, where $A$ is ample and $E$ is effective.
By (\ref{item: eff}), $A\cdot \alpha ,E\cdot \alpha\in \overline{\Eff}_{d-1}(X)$.
In particular, $H\cdot \alpha=0$ implies $A\cdot \alpha=0$ \cite[Corollary 3.8]{FL1}, which can only happen when $\alpha=0$ \cite[Corollary 3.16]{FL1}.

For (\ref{item: nef}), we may again assume $\alpha=\pi_{*}[Z]$, where $\pi:Y\rightarrow X$ is a proper, surjective map and $Z$ a nef subvariety of $Y$.
We also assume $d\geq 2$, otherwise the result already follows from (\ref{item: eff}).
Note that $\pi_{*}\overline{\WMov}_{d-1}(Y)\subseteq\overline{\WMov}_{d-1}(X)$ by the definition of weakly movable cone.
It suffices to show that $H\cdot [Z]\in \overline{\WMov}_{d-1}(Y)$, where $H$ is a very ample divisor on $Y$.
We may assume that $H\cap Z$ is of dimension $d-1$ and is integral \cite[Corollaire 6.11]{Jouanolou}.
By Corollary \ref{cor: intersection of ample}, $H\cap Z$ is a nef subvariety in $Y$.
\end{proof}

\begin{proposition}\label{prop: weakly movable 1}
Let $X$ be a projective variety of dimension $n$ over $k$.
Then
    \[
    \overline{\WMov}_{1}(X)=\overline{\Mov}_{1}(X)
    \]
\end{proposition}
\begin{proof}

Let $\pi:Y\rightarrow X$ be a proper and
surjective map, $Z\subset Y$ be a nef subvariety of dimension $1$.
To show that $\pi_{*}[Z]\in\overline{\Mov}_1(X)$, it suffices to show that $D\cdot \pi_{*}[Z]=\pi^{*}D\cdot [Z]\geq0$ for any pseudoeffective divisor on $X$,
since the dual cone of $\overline{\Mov}_1(X)$ is the cone of pseudoeffective divisors \cite{BDPP}.
This follows from Theorem \ref{theorem: pseudoeffective}.
\end{proof}

Let us recall Hartshorne's Conjecture A:
\begin{conjecture}[{\cite[Conjecture 4.4]{Hartshorne}}]
Let $X$ be a smooth projective variety, and $Y$ be a smooth subvariety with ample normal bundle. Then $n[Y]$ moves in a large algebraic family for $n\gg0$.
\end{conjecture}

This was disproved by Fulton and Lazarsfeld.
They constructed an ample rank $2$ vector bundle on $\mathbb{P}^2$, such that no multiples of the zero section in the total space of the vector bundle moves. 

In view of Proposition \ref{prop: mov vs wmov}, Theorem \ref{theorem: weakly movable} and Proposition \ref{prop: weakly movable 1}, it seems reasonable for us to ask the following
\begin{question}
Let $X$ be a projective variety of dimension $n$. Do we have
\[
\overline{\WMov}_d (X)=\overline{\Mov}_d(X),
\]
for $2\leq d \leq n-1$?
\end{question}

The answer is yes if and only if the cycle class of any nef subvariety of $X$ lies in the movable cone.
The key points in the question,
comparing to Hartshorne's Conjecture A,
are that we only consider the cycle classes up to numerical equivalence;
the movable cone is defined to be the \textit{closure} of the cone generated by movable cycles.
This seems to be one of the weakest possible ways of stating a conjecture that relates positivity of the normal bundle of subvarieties and their movability.

One might want to study the closure of the convex cone generated by the cycle class of nef subvarieties of dimension $d$ (in $\N_d(X)$) instead.
However, the cone may not be of full dimension, when $d=\dim X-1$ and $X$ is singular.
\begin{lemma}[{\cite[Corollary 3.4]{Ottem}}]
Let $X$ be a normal projective variety over $k$ and $Y\subset X$ be a nef subscheme of codimension $1$.
Then $Y$ is a (nef) Cartier divisor.
\end{lemma}
\begin{proof}
Let $\pi:\Bl_Y X\rightarrow X$ be the blowup of $X$ along $Y$, with exceptional divisor $E$.
Then $\pi|_E: E\rightarrow Y$ is equidimensional of relative dimension $0$, by Proposition \ref{prop: equi}.
Therefore, $\pi$ is quasi-finite.
A proper and quasi-finite morphism is finite, so $\pi$ is finite and birational, with $X$ normal.
This implies that $\pi$ is in fact an isomorphism.
\end{proof}
\begin{example}
Let $X$ be a projective variety of dimension $n$ over $k$.
By \cite[Example 19.3.3]{Fulton}, 
the natural map $\N^{1}(X)\xrightarrow{\cdot [X]}\N_{n-1}(X)$ is injective.
Fulger and Lehmann gave an example \cite[Example 2.7]{FL1} where $\N^{1}(X)\xhookrightarrow{\cdot [X]}\N_{n-1}(X)$ is not surjective.
We may assume that $X$ is normal in their example.
By the above lemma, the closure of the convex cone generated by the cycle classes of nef subschemes of codimension $1$ lies in the proper subspace $\N^{1}(X)\subsetneq\N_{n-1}(X)$,
hence is not of full dimension in $\N_{n-1}(X)$.
\end{example}


\end{document}